\documentclass[12pt]{article}

\usepackage{amsfonts,amsmath,amssymb}
\usepackage{dutchcal}
\usepackage{hyperref}
\usepackage{esint}
\usepackage[charter]{mathdesign}
\usepackage{tikz-cd}
\usepackage{xcolor}
\usepackage{enumerate}
\usepackage{epsf,epsfig,amsfonts,a4wide}
\usepackage{amsfonts,mathdots}
\usepackage{amsmath,amssymb,amsthm}
\usepackage{url}
\usepackage{amsmath,amssymb,amsthm}

\DeclareFontFamily{U}{mathx}{}
\DeclareFontShape{U}{mathx}{m}{n}{<-> mathx10}{}
\DeclareSymbolFont{mathx}{U}{mathx}{m}{n}
\DeclareMathAccent{\widehat}{0}{mathx}{"70}
\DeclareMathAccent{\widecheck}{0}{mathx}{"71}

\newtheorem{Theorem}{Theorem}

\newtheorem{Lemma}{Lemma}[section]
\newtheorem{Proposition}{Proposition}[section]
\newtheorem{Corollary}{Corollary}[section]

\newtheorem{Definition}{Definition}[section]

\newtheorem{Remark}{Remark}[section]

\theoremstyle{remark}

\newcommand{\T}{\mathbb{T}}
\newcommand{\C}{\mathbb{C}}
\newcommand{\Z}{\mathbb{Z}}
\newcommand{\1}{\langle}
\newcommand{\2}{\rangle}

\newcommand{\be}{\begin{equation}}
\newcommand{\ee}{\end{equation}}

\newcommand{\R}{\mathbb{R}}

\newcommand{\pd}[2]{\frac{\partial#1}{\partial#2}}

\newcommand{\weg}[1]{}

\usepackage{tikz-cd}

\newcommand{\A}{\mathbb{A}}
\newcommand{\W}{\mathcal{W}}
\newcommand{\ZZ}{\mathcal{Z}}
\newcommand{\HH}{\mathcal{H}}
\newcommand{\D}{\mathbb{D}}
\newcommand{\LL}{\mathcal{L}}
\newcommand{\OO}{\mathcal{O}}

\title{Integrable metrics with potentials on the Hardy space and a class of PDEs}

\author{A. Konyaev and P. Topalov}

\begin{document}

\maketitle

\begin{abstract}
We introduce a new class of quadratic in the momenta Hamiltonians with potential
on Banach spaces of infinite dimension which possess an infinite family of conserved quantities 
in involution with respect to a naturally defined Poisson structure. 
We show that the Hamiltonian flows of the constructed integrals are locally well defined and 
that they are related to a hierarchy of evolution PDEs via nonlinear analytic transformations of 
the phase space, called $Q$-transforms. 
Our framework offers a new geometric interpretation of hierarchies of integrable PDEs, 
linking infinite-dimensional phase spaces directly to the classical 
geometric phenomenon of geodesic equivalence.
\end{abstract}

\section{Introduction}
In this paper, we introduce and study the properties of a novel class of Hamiltonian systems on 
Banach spaces of infinite dimension. These systems possess an infinite family of conserved 
quantities that are in involution with respect to a naturally defined Zakharov-Shabat type 
Poisson bracket. Their Hamiltonian function consists of a quadratic in the momenta kinetic part 
and a potential part that depends only on the position of the point, and on a functional parameter. 
More specifically,  we consider the Hardy space of complex-valued functions on the 1d torus 
$\T:=\R/\Z$,
\[
\HH^{1,+}_0:=\big\{f\in\HH^1\,\big|\, f_k=0\,\,\text{for}\,\,k\le 0\big\},
\]
where $f_k$, $k\in\Z$, are the Fourier coefficients of $f$, and $\HH^1$ is the Hilbert space of
complex-valued functions on $\T$ equipped with the norm  
$\|f\|_{\HH^1}:=\big(\sum_{k\in\Z}|f_k|^2\1k\2^2\big)^{1/2}$, where $\1k\2:=|k|$ for $k\ne 0$ and
$\1 0\2:=1$. Our phase space is the cartesian product $\HH^{1,+}_0\times\HH^{1,+}_0$, where
the first component represents the position of a point, and the second component 
is the momentum. For a given $(w,p)\in\HH^{1,+}_0\times\HH^{1,+}_0$ consider the quadratic in
the momenta Hamiltonian function
\begin{equation}\label{eq:H-introduction}
H(w,p):=\big\1 w,p^2\big\2=\int_0^1\check{w}(y)p(y)^2\,dy
\end{equation}
where $\1 f,g\2=\sum_{k\in\Z}f_kg_k$ for $f,g\in\HH^1$ and $\check{w}(y)=w(-y)$ for $y\in\T$.
Since $\HH^1$ is a Banach algebra, the Hamiltonian function \eqref{eq:H-introduction} is 
well-defined and analytic on $\HH^{1,+}_0\times\HH^{1,+}_0$. 
We refer the reader to Section \ref{sec:phase_space} for the definition of the Poisson bracket on
$\HH^{1,+}_0\times\HH^{1,+}_0$.
For any given $w\in\HH^{1,+}_0$, we define in Section \ref{sec:the_operator} a family of 
bounded operators (cf. \eqref{eq:A_k(H)} and Proposition \ref{prop:A-operator(H)}),
called {\em geodesic operators},
\[
A_k(w) : \HH^{1,+}_0\to\HH^{1,+}_0,\quad k\ge 1,
\]
such that $\big\1 A_k(w)f,g\big\2=\big\1 f,A_k(w)g\big\2$ for any $f,g\in\HH^{1,+}_0$.
For $(w,p)\in\HH^{1,+}_0\times\HH^{1,+}_0$ we then set (cf. \eqref{eq:I_k})  
\begin{equation}\label{eq:I_k-introduction}
I_k(w,p):=\big\1 A_k(w)\Gamma(w)p,p\big\2,\quad k\ge 1,
\end{equation}
where $\Gamma(w)$ is the Hankel operator on $\HH^{1,+}_0$ with symbol $w$,
\begin{equation}\label{eq:Hankel-introduction}
\Gamma(w) : \HH^{1,+}_0\to\HH^{1,+}_0,\quad p\mapsto\Pi^+_0(w\check{p}),
\end{equation}
and $\Pi^+_0 : \HH^1\to\HH^{1,+}_0$ is the $\HH^1$-orthogonal (Szeg\"o) projection onto $\HH^{1,+}_0$ 
(cf. \eqref{eq:Pi(H)}). The following Theorem is proved in Section \ref{sec:involutivity_f=0}.

\begin{Theorem}\label{th:involutivity}
\begin{itemize}
\item[(i)] The functions $I_k : \HH^{1,+}_0\times\HH^{1,+}_0\to\C$, 
$k\ge 1$, are analytic integrals of the Hamiltonian flow of $H$.
\item[(ii)] $\{I_k,I_l\}=0$ for any $k,l\ge 1$.
\end{itemize}
\end{Theorem}

\noindent  An explicit formula for the integrals \eqref{eq:I_k-introduction} is given
in Proposition \ref{prop:I-formulas}. 
It is worth noting that our Poisson structure is {\em weak} on $\HH^{1,+}_0\times\HH^{1,+}_0$, 
implying that the Hamiltonian field $X_F$ of a given $C^1$-function 
$F : \HH^{1,+}_0\times\HH^{1,+}_0\to\C$ does not necessarily take values in 
$\HH^{1,+}_0\times\HH^{1,+}_0$ but rather in a larger space. Remarkably, this does not happen 
with the integrals \eqref{eq:I_k-introduction}, and they have well-defined analytic Hamiltonian 
vector fields $X_{I_k} : \HH^{1,+}_0\times\HH^{1,+}_0\to\HH^{1,+}_0\times\HH^{1,+}_0$, $k\ge 1$
(cf. Corollary \ref{coro:X_I_k-vector-field(H)}). In particular, the corresponding Hamiltonian flows
are locally well-defined.

\noindent{\em The case of potentials.}
One of the core results of this paper is the existence of a class of potentials
$V_f : \OO\to\C$, where $\OO$ is an open dense subset in $\HH^{1,+}_0\times\HH^{1,+}_0$ and 
$f\in\HH^1$ is a parameter, such that the Hamiltonian function
\begin{equation}\label{eq:H_f-introduction}
H_f(w,p):=H(w,p)+V_f(w),\quad(w,p)\in\OO\times\HH^{1,+}_0,
\end{equation}
possesses an infinite family of integrals in involution.
More specifically, for any given $f\in\HH^1$ and $w\in\OO$ we define
\begin{equation}\label{eq:V_f-introduction}
V_f(w):=\1 1,f/w\2=\int_0^1\frac{f(y)}{w(y)}\,dy,
\end{equation}
where $\OO:=\big\{w\in\HH^{1,+}_0\,\big|\,w(y)\ne 0\,\,\forall y\in\T\big\}$ is an open dense set
in $\HH^{1,+}_0$. In addition to \eqref{eq:H_f-introduction}, for any $k\ge 1$ consider 
the {\em individual potentials}
\begin{equation}\label{eq:U^f_k-introduction}
U_k^f(w):=-\big(w\Pi^+_0(f/w)\big)_{k+1},\quad w\in\OO,
\end{equation}
where $(g)_l$ denotes the $l$-th Fourier coefficient of $g\in\HH^1$, and then set
\begin{equation}\label{eq:J^f_k-introduction}
J^f_k(w,p):=I_k(w,p)+U^f_k(w),\quad(w,p)\in\OO\times\HH^{1,+}_0,
\end{equation}
for any $k\ge 1$.
Let us briefly comment on formula \eqref{eq:U^f_k-introduction}.
Since $\HH^1$ is a Banach algebra, the product of $w$ and the Szeg\"o projection 
$\Pi^+_0(f/w)$ of $f/w$ belong to $\HH^1$, and hence, the individual potential $U^f_k(w)$ 
is well-defined and equals the $(k+1)$-th Fourier coefficient of the product $w\Pi^+_0(f/w)$.

The following result extends Theorem \ref{th:involutivity} to the case of potentials.

\begin{Theorem}\label{th:J_k-involutivity}
Take $f\in\HH^1$. Then we have:
\begin{itemize}
\item[(i)] The functions $J_k^f : \OO\times\HH^{1,+}_0\to\C$, 
$k\ge 1$,
are analytic integrals of the Hamiltonian flow of $H_f$.
\item[(ii)] $\{J_k^f,J_l^f\}=0$ for any $k,l\ge 1$.
\end{itemize}
\end{Theorem}

\noindent As in the case of the integrals \eqref{eq:I_k-introduction} of $H$, the integrals
\eqref{eq:J^f_k-introduction} of $H_f$ have well-defined Hamiltonian vector fields 
$X_{J^f_k} : \OO\times\HH^{1,+}_0\to\HH^{1,+}_0\times\HH^{1,+}_0$, $k\ge 1$ 
(cf. Corollary \ref{coro:X_J_k-vector-field(H)}). In particular, the corresponding Hamiltonian flows
are locally well-defined (cf. Theorem \ref{th:existence}).

\medskip

\noindent {\em The momentum map.} Let us now fix the parameter function $f\in\HH^1$ and
assign to any $(w,p)\in\OO\times\HH^{1,+}_0$ the complex-valued sequence
$\big(J^f_k(w,p)\big)_{k\ge 1}$ formed by the values of the integrals \eqref{eq:J^f_k-introduction},
\begin{equation}\label{eq:pre-J^f-introduction}
J^f : (w,p)\mapsto\big(J^f_k(w,p)\big)_{k\ge 1}\,.
\end{equation}
This is the {\em momentum map} of the Poisson commuting integrals \eqref{eq:J^f_k-introduction}.
We will see in Section \ref{sec:potentials} that for any given $(w,p)\in\OO\times\HH^{1,+}_0$
the sequence $\big(J^f_k(w,p)\big)_{k\ge 1}$ belongs
to the Hilbert space $\mathfrak{h}^1$ of complex-valued sequences $(x_k)_{k\ge 1}$ with finite norm 
$\|x\|_{\mathfrak{h}^1}:=\big(\sum_{k\ge 1}|x_k|^2\1 k\2^2\big)^{1/2}$.
In this way, we have a well-defined map
\begin{equation}\label{eq:J^f-introduction}
\OO\times\HH^{1,+}_0\to\HH^{1,+}_0,\quad (w,p)\mapsto J^f(w,p),
\end{equation}
where we identify $J^f(w,p)$ with the element of $\HH^{1,+}_0$ whose
$k$-th Fourier coefficient is $J^f_k(w,p)$, $k\ge 1$.
We have the following Proposition.

\begin{Proposition}\label{prop:J^f-introduction}
The momentum map \eqref{eq:J^f-introduction} is analytic.
\end{Proposition}

\noindent A similar statement holds for the momentum map
of the Benjamin-Ono equation -- see \cite[Theorem 3]{GKT1} 
as well as \cite{GKT,GT,Atiyah} for the context. 
Note that in contrast to \cite[Theorem 3]{GKT1}, here we derive an explicit formula for
the momentum map (cf. \eqref{eq:J(w,p)} in Section \ref{sec:potentials}).

\medskip

\noindent{\em The $Q$-transforms.} The $Q$-transforms are nonlinear maps that 
transform the Hamiltonian flows of the integrals \eqref{eq:J^f_k-introduction} on the
phase space $\OO\times\HH^{1,+}_0$ onto solutions of a class of nonlinear PDEs. 
Here we briefly discuss these maps (see Section \ref{sec:PDEs} for details).
For any given value of the parameter $f\in\HH^1$, consider the commutative diagram
\begin{equation}\label{eq:Q-diagram}
\begin{tikzcd}
&\OO\times\HH^{1,+}_0\arrow[ld, "\mathcal{Q}_f^-", swap]\arrow[rd, "\mathcal{Q}_f^+"]&\\
\HH^{1,-}_0\arrow[rr, dashrightarrow]& &\HH^{1,+}
\end{tikzcd}
\quad\quad
\begin{tikzcd}
&(w,p)\arrow[ld, "\mathcal{Q}_f^-", swap, mapsto]\arrow[rd, "\mathcal{Q}_f^+", mapsto]&\\
-\Pi^-_0\big(f/w^2\big)\arrow[rr, dashrightarrow, mapsto]& &\Pi^+\big(f/w^2\big)
\end{tikzcd}
\end{equation}
where 
\[
\HH^{1,+}=\big\{f\in\HH^1\,\big|\, f_k=0\,\,\text{for}\,\,k<0\big\}
\quad\text{\rm and}\quad
\HH^{1,-}_0=\big\{f\in\HH^1\,\big|\, f_k=0\,\,\text{for}\,\,k\ge 0\big\},
\]
and $\Pi^+ : \HH^1\to\HH^{1,+}$ and $\Pi^-_0 : \HH^1\to\HH^{1,-}_0$ are the corresponding 
$\HH^1$-orthogonal projections onto  $\HH^{1,+}$, respectively, $\HH^{1,-}_0$ 
(cf. \eqref{eq:Pi(H)} in Section \ref{sec:spaces}). 
The maps $\mathcal{Q}_f^\pm$ are the {\em Q-transforms} -- see \eqref{eq:Q+} and \eqref{eq:Q-}.
The dashed horizontal map is defined only on the image of $Q_f^-$ and, generally, 
depends on the choice of $w\in\OO$.

Let us now briefly explain how the diagram \eqref{eq:Q-diagram} is related to a class of
nonlinear PDEs. We refer to Section \ref{sec:PDEs} for the details.
It follows from Theorem \ref{th:J_k-involutivity} and the local existence of the 
Hamiltonian flows of $H_f$ and $J^f_k$, $k\ge 1$, that for any given index $k_0\ge 1$ and
$(w_0,p_0)\in\OO\times\HH^{1,+}_0$ there exist $T>0$ and a $C^\ell$-map, $\ell\ge 3$,
\begin{equation}\label{eq:joint_solution}
[0.T)\times[0,T)\to\OO\times\HH^{1,+}_0,\quad(x,t)\mapsto\big(w(x,t),p(x,t)\big),
\end{equation}
which is the {\em joint flow} of the Hamiltonian functions $H_f$ and $J^f_{k_0}$ with initial
data $(w_0,p_0)$, i.e., we have that $(w,p)|_{(x,t)=(0,0)}=(w_0,p_0)$, and for any given 
$t\in[0,T)$ the curve $[0,T)\to\OO\times\HH^{1,+}_0$, $x\mapsto\big(w(x,t),p(x,t)\big)$,
is the unique solution of the Hamiltonian equation with Hamiltonian $H_f$ and 
initial data $\big(w(0,t),p(0,t)\big)$, and for any given $x\in[0,T)$ the curve 
$[0,T)\to\OO\times\HH^{1,+}_0$, $t\mapsto\big(w(x,t),p(x,t)\big)$,
is the unique solution of the Hamiltonian equation with Hamiltonian $J^f_{k_0}$
and initial data $\big(w(x,0),p(x,0)\big)$. 
The $Q$-transform of \eqref{eq:joint_solution},
\[
q^\pm=\mathcal{Q}^+_f(w,p),\quad[0,T)\times[0,T)\to\HH^{1,\pm},
\]
is then a solution of a system of PDEs (see Theorem \ref{th:(w,q)-evolution}).
By choosing different values of the parameter $f\in\HH^1$ and $k_0\ge 1$, 
we obtain solutions of a large class of nonlinear PDEs. 
We illustrate this procedure in Section \ref{sec:PDEs},
where we generate solutions of systems of KdV and BKM equations. 
Note that in this way one can generate solutions of a much larger class of nonlinear equations, 
including the Camassa-Holm, the Harry Dym, and a number of other equations.

\medskip

The kinetic part \eqref{eq:H-introduction} of the Hamiltonian function 
\eqref{eq:H_f-introduction} can be considered as an instance of an infinite-dimensional 
version of a class of Riemannian metrics with completely integrable geodesic flows, 
called {\em projectively} or, equivalently, {\em geodesically} equivalent metrics. 
These metrics were introduced by Levi-Civita and have been extensively studied 
(cf. \cite{Levi-Civita,MT1,T1,BKM1}). 
We discuss this relation in Section \ref{sec:finite_dynamics}, where we prove that 
for specifically chosen initial data, an infinite-dimensional analog of the Hamiltonian reduction 
reduces the Hamiltonian flow of \eqref{eq:H-introduction} to the geodesic flow of a metric that 
allows a geodesic equivalence. We expect that the indicated correspondence between the 
rich theory of geodesically equivalent metrics and their possible infinite-dimensional 
analogs could uncover, via the corresponding $Q$-transforms, new properties of integrable PDE's. 
As mentioned above, the Hamiltonian flows of the integrals \eqref{eq:J^f_k-introduction} are 
locally well-defined on the phase space $\HH^{1,+}_0\times\HH^{1,+}_0$ and depend smoothly on 
the initial data (Theorem \ref{th:existence}). In this way, by applying the $Q$-transform, 
we obtain solutions of a class of non-linear PDE's. This geometric approach bears similarities to
the method of characteristics for solving first order PDE's and can be considered as 
its extension to a larger class of equations. 
In contrast to the dressing action on loop groups (cf., e.g., \cite{Palais} and 
the references therein), our phase space is a direct infinite-dimensional analog of 
the co-tangent bundle of a manifold equipped with the canonical Poisson structure. 
The Hamiltonian function and the integrals are quadratic in momenta with a potential and, 
in view of their simplicity, can be regarded as preliminary versions of action coordinates 
(cf., e.g., \cite{GKT,GT}).
Note that finite dimensional flat metrics appear in the framework of the systems of 
hydrodynamic type \cite{balinskii, dn, ferapontov4, mokhov4, mokhov2, maltsev}. 
Such systems have been around for quite some time: apparently they first appeared in
the description of superfluid helium by Landau \cite{landau}. 
For the BKM equations, we refer to \cite{BKM1,app_nij4,finite1}.
In a different context, formal difference expressions similar to the geodesic 
operators studied in Section \ref{sec:the_operator} appear in \cite{Sh1}.

Let us also discuss our choice of function spaces. The Hardy space $\HH^{1,+}_0$ (as well as
the supplementary space $\W^{1,+}_r$, $r>1$, of holomorphic functions defined in 
Section \ref{sec:spaces}) appears naturally in our framework since the kinetic part 
\eqref{eq:H-introduction} of the Hamiltonian function \eqref{eq:H_f-introduction} is 
related to the Hankel operator \eqref{eq:Hankel-introduction}. 
We do not make an effort to lower the regularity exponent of the spaces
since we prefer to work with Banach algebras and analytic maps, thus keeping the exposition 
as transparent as possible. However, note that all proved results continue to hold 
for the space $\HH^{s,+}_0$ with regularity exponent $s>1/2$.
The results also hold for the supplementary space $\W^{1,+}_{r,0}$, $r>1$.
We expect that the class of potentials \eqref{eq:V_f-introduction} can be enlarged if 
the dynamics is restricted to specific subspaces of entire functions of
the Hardy space $\HH^{1,+}_0$. For simplicity, we restrict ourselves to spaces over $\C$.
The various restrictions of the Hamiltonian dynamics to real subspaces require separate 
consideration since they are related to the important question of the global in time existence of 
the obtained Hamiltonian flows.

\section{Spaces of analytic functions}\label{sec:spaces}
In this Section we introduce the function spaces used in the paper and 
prove several preliminary results needed in the sequel. For given Banach spaces $X$ and $Y$ denote 
by $\LL(X,Y)$ the Banach space of bounded linear maps $X\to Y$.
If $X=Y$ we set $\LL(X):=\LL(X,X)$.

\noindent{\em The space $\W^1_r$.}
We first define a Hilbert spaces of holomorphic function on the annulus
\[
\A_r:=\big\{z\in\C\,\big|\,1/r<|z|<r\big\},\quad r>1.
\]
These spaces will play a supplementary role in our analytic framework.
For a given $r>1$ denote by $\W_r^1$ the linear space of holomorphic functions $f : \A_r\to\C$ 
on $\A_r$ whose Laurent coefficients\footnote{Alternatively, $f_k$ is the $k$-th Fourier coefficient
of the restriction $f(e^{2\pi i y})$, $y\in[0,1]$.}
\begin{equation}\label{eq:f_k}
f_k:=\frac{1}{2\pi i}\oint_{|z|=1}\frac{f(z)}{z^{k+1}}\,dz=\int_0^1f(e^{2\pi i y})\,e^{-2\pi i k y}\,dy,
\quad k\in\Z,
\end{equation}
have finite norm
\begin{equation}\label{eq:W1-norm}
\|f\|_{\W^1_r}:=\Big(\sum_{k\in\Z}|f_k|^2r^{2|k|}\1 k\2^2\Big)^{1/2}<\infty\,.
\end{equation}
(The cycle of integration in \eqref{eq:f_k} is oriented counterclockwise.)
We equip $\W^1_r$ with the norm \eqref{eq:W1-norm} and consider the scalar product
\begin{equation}\label{eq:W-scalar_product}
(f,g)_{\W^1_r}:=\sum_{k\in\Z}f_k\overline{g}_k r^{2|k|}\1 k\2^2,
\quad f,g\in\W^1_r,
\end{equation}
where, as in the Introduction,
\[
\1 k\2=\left\{
\begin{array}{l}
|k|,\quad k\ne 0,\\
1,\quad k=0,
\end{array}
\right.
\]
and
\begin{equation}\label{eq:L2-pairing}
\1 f,g\2=\sum_{k\in\Z} f_k g_k\quad f,g\in\W^1_r .
\end{equation}
It is clear from the Cauchy-Schwarz inequality that the sesquilinear form \eqref{eq:W-scalar_product} 
and the bilinear form \eqref{eq:L2-pairing} are continuous. In addition to the space $\W^1_r$ we will 
need the closed subspaces
\[
\W^{1,+}_r:=\big\{f\in\W^1_r\,\big|\,f_k=0\,\,\,\text{\rm for}\,\,\,k<0 \big\}\quad
\text{and}\quad\W^{1,-}_r:=\big\{f\in\W^1_r\,\big|\,f_k=0\,\,\,\text{for}\,\,\,k>0 \big\}
\]
as well as the closed subspace $\W^{1,+}_{r,0}$ [and $\W^{1,-}_{r,0}$] in $\W^{1,+}_r$ 
[respectively $\W^{1,-}_r$] with vanishing $0$-th Laurent coefficient $f_0=0$.
We equip these subspaces with the restricted norm \eqref{eq:W1-norm}.
We then have the splitting
\[
\W^1_r=\W^{1,-}_r\oplus\W^{1,+}_{r,0}\quad\text{\rm and}\quad\W^1_r=\W^{1,-}_{r,0}\oplus\W^{1,+}_r
\]
with continuous projections
\begin{equation}\label{eq:Pi}
\Pi^\pm : \W^1_r\to\W^{1,\pm}_r\quad\text{and}\quad\Pi^\pm_0 : \W^1_r\to\W^{1,\pm}_{r,0}\,.
\end{equation}
Note that the elements of $\W^{1,+}_r$ are holomorphic functions on the disk 
\[
\D_r:=\big\{z\in\C\,\big|\,|z|<r\big\}
\]
that are continuous on its closure in $\C$. Similarly, the elements of $\W^{1,-}_r$ are holomorphic 
functions on $\C\setminus\D_{1/r}$ that are continuous on its closure in $\C$.
We have the following simple Lemma. 

\begin{Lemma}\label{lem:W-space}
Take $r>1$. Then we have:
\begin{itemize}
\item[(i)] When equipped with the scalar product \eqref{eq:W-scalar_product} the space
of holomorphic functions $\W^1_r$ is a Hilbert space.
\item[(ii)] The pointwise multiplication
of functions
\[
\W^1_r\times\W^1_r\to\W^1_r,\quad (f,g)\mapsto fg,
\]
is well-defined and continuous. In particular, $\W^1_r$ is a Banach algebra.
\item[(iii)] Any element $f\in\W^1_r$ can be (uniquely) represented as an absolutely and 
uniformly convergent Laurent series
\[
f(z)=\sum_{k\in\Z}f_kz^k=\sum_{k\ge 1}\frac{f_{(-k)}}{z^k}+
\sum_{k\ge 0}f_kz^k
\]
on $\A_r$. In addition, any $f\in\W^1_r$ extends to a continuous function on
the closure $1/r\le|z|\le r$ of $\A_r$ in $\C$ so that
\[
\max_{1/r\le|z|\le  r}|f(z)|\le C\|f\|_{\W^1_r}
\]
where the constant $C>0$ is independent of the choice of $f\in\W^1_r$.
\end{itemize}
\end{Lemma}

The Lemma in proved in Appendix.

The inequality in (iii) implies that the space $\W^1_r$ is continuously embedded in
the space of continuous functions on $1/r\le|z|\le r$. It follows from 
Lemma \ref{lem:A-operator} (ii) that $\W^{1,+}\subseteq\W^1_r$ and 
$\W^{1,+}_0\subseteq\W^1_r$ are a Banach subalgebras in $\W^1_r$.

For any $f\in\W^1_r$ and $z\in\A_r$ denote $\check{f}(z):=f(1/z)$.
Since $\check{f}_k=f_{-k}$, $k\in\Z$, we have that 
$\|\check{f}\|_{\W^1_r}=\|f\|_{\W^1_r}$ and hence the bijective map
\begin{equation}\label{eq:check_complex}
\W^1_r\to\W^1_r,\quad f\mapsto\check{f},
\end{equation}
is a (complex) linear isometry. 
Note that the pairing \eqref{eq:L2-pairing} can then be written as
\begin{equation}\label{eq:L2-pairing(W)}
\1 f,g\2=\frac{1}{2\pi i}\oint_{|z|=1} f(z)\,\check{g}(z)\,\frac{dz}{z},\quad f,g\in\W^1_r,
\end{equation}
where the cycle of integration is oriented counterclockwise.
Clearly, we have that
\begin{equation}\label{eq:check-move}
\1f,gh\2=\1 f\check{g},h\2
\end{equation}
for any $f,g,h\in\W^1_r$.

\medskip

\noindent{\em The Hardy space $\HH^{1,+}$.}
Denote by $L^2$ the space of square integrable complex-valued functions on the torus
$\T=\R/\Z$. As in the Introduction, consider the Sobolev space $\HH^1\subseteq\ L^2$ 
of complex-valued functions on $\T$ whose Fourier coefficients
\[
f_k:=\int_0^1f(y)\,e^{-2\pi i k y}\,dy
\]
have finite norm
\begin{equation}\label{eq:H1-norm}
\|f\|_{\HH^1}=\Big(\sum_{k\in\Z}|f_k|^2\1 k\2^2\Big)^{1/2}<\infty.
\end{equation}
When equipped with the scalar product
\begin{equation}\label{eq:H1-metric}
(f,g)_{\HH^1}:=\sum_{k\in\Z} f_k\overline{g}_k\1 k\2^2
\end{equation}
the space $\HH^1$ is a Hilbert space that is continuously embedded in $C(\T,\C)$. 
Moreover, the pointwise multiplication of functions
\[
\HH^1\times\HH^1\to\HH^1,\quad (f,g)\mapsto fg,
\]
is continuous. Hence, $\HH^1$ is a Banach algebra. In addition to $\HH^1$ we will also need 
the corresponding Hardy spaces
\[
\HH^{1,+}:=\big\{f\in\HH^1\,\big|\,f_k=0\,\,\text{for}\,\,k<0\big\}\quad\text{and}\quad
\HH^{1,-}:=\big\{f\in\HH^1\,\big|\,f_k=0\,\,\text{for}\,\,k> 0\big\}
\]
as well as the subspace $\HH^{1,+}_0$ [and $\HH^{1,-}_0$ ] in $\HH^{1,+}$ 
[respectively $\HH^{1,-}$ ] with vanishing $0$-th Fourier coefficient $f_0=0$.
In a similar way as above (see \eqref{eq:Pi}) we have the splitting
\[
\HH^1=\HH^{1,-}\oplus\HH^{1,+}_0\quad\text{and}\quad\HH^1=\HH^{1,-}_0\oplus\HH^{1,+}
\]
as well the continuous projections
\begin{equation}\label{eq:Pi(H)}
\Pi^\pm : \HH^1\to\HH^{1,\pm}\quad\text{and}\quad
\Pi^\pm_0 : \HH^1\to\HH^{1,\pm}_0\,.
\end{equation}
The spaces $\HH^{1,\pm}$ and $\HH^{1,\pm}_0$ are Banach subalgebras in $\HH^1$.
If with the help of the map $\T\to\mathbb{S}$, $y\mapsto e^{2\pi i y}$,
we identify the torus $\T$ with the circle $\mathbb{S}:=\{z\in\C\,|\,|z|=1\}$ in $\C$ 
we see from the definitions of the norms \eqref{eq:W1-norm} and \eqref{eq:H1-norm} that 
for any $r>0$ the space of holomorphic functions $\W^1_r$ is densely and 
continuously embedded in the Sobolev space $\HH^1$,
\[
\W^1_r\subseteq\HH^1\,.
\]
More specifically, the embedding is given by the map
\begin{equation}\label{eq:W->H}
\W^1_r\hookrightarrow\HH^1,\quad f\mapsto [y\mapsto f(e^{2\pi i y})].
\end{equation}
As in the case of $\W^{1,+}_r$, the elements of $\HH^{1,+}$ can be identified with 
holomorphic functions on the unit disk $\D$ by the formula
\begin{equation}\label{eq:f<->f(z)}
\HH^{1,+}\ni f\mapsto\sum_{k\ge 0}f_k\,z^k,\quad z\in\D\,.
\end{equation}
Since $f\in\HH^{1,+}$, the series converges uniformly on the closure of $\D$,
and hence the holomorphic function in \eqref{eq:f<->f(z)} is continuous on 
the closure of $\D$. 
In a similar way, we identify the elements of $\HH^{1,-}$ with holomorphic
functions on $\{z\in\C\,|\,|z|>1\}$,
\begin{equation}\label{eq:f<->f(1/z)}
\HH^{1,-}\ni f\mapsto\sum_{k\ge 0}\frac{f_{(-k)}}{z^k}
\end{equation}
where the series converge uniformly on $\C\setminus\D$.
More generally, for $f\in\HH^1$ and $z\in\mathbb{S}$ we set
\begin{equation}\label{eq:f(z)}
f(z):=\sum_{k\ge 1}\frac{f_{(-k)}}{z^k}+\sum_{k\ge 0}f_k\,z^k\,.
\end{equation}
It is clear that $\W^{1,\pm}_r\subseteq\HH^{1,\pm}$ and
$\W^{1,\pm}_{r,0}\subseteq\HH^{1,\pm}_0$ where the embeddings are dense and continuous.
In view of the embeddings, the operators and the bilinear forms defined on $\W^1_r$ extend 
by continuity to operators and forms on $\HH^1$ that will be denoted by the same letters.
In fact, the involution \eqref{eq:check_complex} extends to an isometry 
\[
\HH^1\to\HH^1,\quad f\mapsto\check{f},
\]
where $\check{f}(y)=f(-y)$. The pairing \eqref{eq:L2-pairing} extends by continuity
to a bounded bilinear form on $\HH^1$,
\begin{equation}\label{eq:L2-pairing(H)}
\1 f,g\2=\sum_{k\in\Z}f_k g_k=\int_0^1f(y)\check{g}(y)\,dy,\quad f,g\in\HH^1,
\end{equation}
that satisfies \eqref{eq:check-move}.

\medskip

\noindent The Banach algebras $\W^1_r$, $\W^{1,+}_r$, and $\HH^{1,+}$, satisfy 
the following important property.

\begin{Lemma}\label{lem:1/w}
Take $r>1$. Then we have:
\begin{itemize}
\item[(i)] If $w\in\W^{1,+}_r$ does not have zeros in the closure of the disk $\D_r$ then
the function $1/w$ belongs to $\W^{1,+}_r$.
\item[(ii)] If $w\in\W^1_r$ does not have zeros in the closure of the annulus $\A_r$ then
the function $1/w$ belongs to $\W^1_r$.
\item[(iii)] If $w\in\HH^1$ does not have zeros on the torus $\T$ then
the function $1/w$ belongs to $\HH^1$.
\item[(iv)] If $w\in\HH^{1,+}$ does not have zeros in the closure of the disk $\D$ 
(see the identification \eqref{eq:f<->f(z)}) 
then the function $1/w$ belongs to $\HH^{1,+}$.
\end{itemize}
\end{Lemma}

\noindent Note that Lemma \ref{lem:1/w} (iii) does not hold with $\HH^1$ replaced by $\HH^{1,+}$.
In fact, if $w(y):=e^{2\pi i y}$ then $w\in\HH^1$, $w$ does not vanish on $\T$,
an $1/w\in\HH^{1,-}$.

\begin{proof}[Proof of Lemma \ref{lem:1/w}]
The Lemma follows easily from the properties of commutative Banach algebras.
Recall that a bounded linear functional $\alpha : X\to\C$ on a Banach algebra $X$ is called
multiplicative if $\alpha(fg)=\alpha(f)\alpha(g)$ for any $f,g\in X$.
By Lemma \ref{lem:W-space} the space $\W^{1,+}_r$ is a commutative Banach algebra.
Let us prove item (i). Fix $r>1$ and assume that $w\in\W^{1,+}_r$ does not have zeros in
the closure of $\D_r$. 
Let $\alpha : \W^{1,+}_r\to\C$, $\alpha\ne 0$, be a multiplicative bounded linear functional
on $\W^{1,+}_r$. Then, $\alpha$ is a contraction, i.e.,
$|\alpha(f)|\le\|f\|_{\W^1_r}$ for any $f\in\W^{1,+}_r$.
By setting $\ZZ(z):=z$ for $z\in\D_r$ we then obtain from the definition of 
the norm \eqref{eq:W1-norm} that
$|\alpha(\ZZ)|\le\|\ZZ\|_{\W^1_r}=r$.
This implies that 
\begin{equation}\label{eq:alpha1}
\alpha(\ZZ)=\omega
\end{equation}
for some $\omega\in\C$ such that $|\omega|\le r$.
Any element $w\in\W^{1,+}_r$ can be represented as
$w=\sum_{k\ge 0}w_k\ZZ^k$
where, by the definition of the norm \eqref{eq:W1-norm}, the series converges 
absolutely in $\W^{1,+}_r$. We then obtain that
\[
\alpha(w)=\sum_{k\ge 0}w_k\alpha(\ZZ)^k=w(\omega).
\]
Hence, any non-identically equal to zero multiplicative bounded liner functional $\alpha$ on
$\W^{1,+}_r$ is of the form $\alpha(w)=w(\omega)$ for some $\omega$ in the closure of the disk $\D_r$. 
Then, the condition that $w\in\W^{1,+}_r$ does not have zeros in the closure of $\D_r$ implies that 
$\alpha(w)\ne 0$ for all non-identically equal to zero multiplicative bounded linear functionals 
on $\W^{1,+}_r$. By Gelfand's theory of commutative Banach algebras (\cite{Gelfand}), 
we then conclude that $w$ is an invertible element of $\W^{1,+}_r$, and hence $1/w\in\W^{1,+}_r$.
This proves item (i). Let us now prove (ii). 
Let $\alpha : \W^1_r\to\C$, $\alpha\ne 0$, be a multiplicative 
bounded linear functional on $\W^1_r$. By arguing in the the same way as in proof of (i) 
we then conclude that $\alpha(\ZZ)=\omega$ where $|\omega|\le r$. We have that
$\ZZ^{-1}\ZZ=1$ where $\ZZ^{-1}(z):=1/z$, $\ZZ^{-1}\in\W^1_r$.
This implies that $\alpha(\ZZ)=\omega\ne 0$ and that 
\begin{equation}\label{eq:alpha2}
\alpha(\ZZ^{-1})=\alpha(\ZZ)^{-1}=1/\omega\,.
\end{equation}
Since $\alpha$ is a contraction, we have that
$|\alpha(\ZZ^{-1})|\le\|\ZZ^{-1}\|_{\W^1_r}\le r$.
By combining this with \eqref{eq:alpha2} we then conclude that 
$|\omega|\ge 1/r$, and hence $1/r\le|\omega|\le r$.
For any element $w\in\W^1_r$ we have that 
$w=\sum_{k\ge 1}w_{-k}(\ZZ^{-1})^k+\sum_{k\ge 0} w_k\ZZ^k$. 
It then follows from \eqref{eq:alpha1} and \eqref{eq:alpha2} that
\begin{equation}\label{eq:alpha3}
\alpha(w)=\sum_{k\ge 1}w_{-k}\alpha(\ZZ^{-1})^k+
\sum_{k\ge 0} w_k\alpha(\ZZ)^k=\sum_{k\in\Z}w_k\omega^k=w(\omega).
\end{equation}
Hence, any non-identically equal to zero multiplicative bounded linear functional on 
$\W^1_r$ is of the form \eqref{eq:alpha3} for some $\omega$ in the closure of the annulus $\A_r$.
As in item (i), item (ii) then follows from Gelfand's theory of 
commutative Banach algebras. Items (iii) and (iv) are well-known and can be proved in 
the same way as (i).
\end{proof}

\begin{Remark}
Note that in the definition of a Banach algebra $(X,\|\cdot\|)$ we allow
the constant $C>0$ in the inequality $\|xy\|\le C\|x\|\|y\|$, $x,y\in X$, 
to be not necessarily equal to one. The results on commutative Banach algebras used
in the proof of Lemma \ref{lem:1/w} continue to hold under such an assumption.
\end{Remark}

\noindent In addition to the spaces above, we will also need the spaces of distributions on the torus
\[
\mathcal{D}^+:=\big\{f\in\mathcal{D}'(\T)\,\big|\, f_k=0\,\,\text{for}\,\,k<0\big\}
\quad\text{and}\quad
\mathcal{D}^+_0:=\big\{f\in\mathcal{D}'(\T)\,\big|\, f_k=0\,\,\text{for}\,\,k\le 0\big\}
\]
where $f_k$ are the Fourier coefficients of $f\in\mathcal{D}'(\T)$.

\section{A family of geodesic operators}\label{sec:the_operator}
In this Section we study the properties of a family of operators 
that plays a central a role in our construction. 
For a given $w,f\in\W^1_r$ and $\mu\in\A_r$ consider the expression
\begin{equation}\label{eq:A-operator}
[A(\mu,w)f](z):=\frac{1}{z-\mu}\big(w(z)f(\mu)-w(\mu)f(z)\big)
\end{equation}
for $z\in\A_r$. Note that for any given $w,f\in\W^1_r$ and $\mu\in\A_r$ the map
$z\mapsto[A(\mu,w)f](z)$, $\A_r\to\C$, is holomorphic and continuous on the closure of $\A_r$.
In this sense, for any given $w\in\W^1_r$ and $\mu\in\A_r$, 
we obtain a complex linear map, denoted $A(\mu,w)$, that assigns a holomorphic function 
$A(\mu,w)f$ on $\A_r$ to any given holomorphic function $f\in\W^1_r$.
The function $w$ will be called the {\em symbol} of the geodesic operator $A(\mu,w)$.
In the case when $f,w\in\W^{1,+}_r$ we will assume that $z,\mu\in\D_r$.
Then, the map $z\mapsto[A(\mu,w)f](z)$, $\D_r\to\C$, is holomorphic and
continuous on the closure of $\D_r$.
In the remainder of this Section we study the analytic properties of
the linear map $A(\mu,w)$.

For simplicity of notation, for $g\in\W^1_r$ and $z\in\A_r$ we
set $\Pi^\pm(g(z)):=[\Pi^{\pm}g](z)$ where $\Pi^\pm : \W^1_r\to\W^{1,\pm}_r$ are 
the (continuous) projections \eqref{eq:Pi}.
We have the following Lemma.

\begin{Lemma}\label{lem:A-operator}
Take $r>1$. Then we have:
\begin{itemize}
\item[(i)] For any given $w,f\in\W^1_r$ and $z\in\A_r$ the map
$\mu\mapsto[A(\mu,w)f](z)$, $\A_r\to\C$, is holomorphic, continuous 
on the closure of $\A_r$, and have Laurent coefficients
\begin{equation}\label{eq:A_k}
[A_k(w)f](z)=\left\{
\begin{array}{l}
f(z)\,\Pi^+(w(z)/z^{k+1})-w(z)\,\Pi^+(f(z)/z^{k+1}),\quad k\ge 0,\\
w(z)\,\Pi^-_0(f(z)\,z^{|k+1|})-f(z)\,\Pi^-_0(w(z)\,z^{|k+1|}),\quad k\le -1.
\end{array}
\right.
\end{equation}
In the case when $w,f\in\W^{1,+}_r$ the map $\mu\mapsto[A(\mu,w)f](z)$, 
$\D_r\to\C$, is holomorphic, continuous on the closure of $\D_r$, 
and the Laurent coefficients $[A_k(w)f](z)=0$ for any $z\in\D_r$ and $k\le -1$.
\item[(ii)]
For any given $k\in\Z$ and $w\in\W^1_r$ the linear operator 
$\W^1_r\to\W^1_r$, $f\mapsto A_k(w)f$, is well-defined and bounded.
Moreover, the map 
\begin{equation}\label{eq:A_k-map}
\W^1_r\to\LL(\W^1_r),\quad w\mapsto A_k(w),
\end{equation}
is analytic. The statement also holds with $\W^1_r$ replaced by $\W^{1,+}_r$.
\end{itemize}
\end{Lemma}

\noindent Note that in view of Hartogs' theorem the map
$\A_r\times\A_r\to\C$, $(\mu,z)\mapsto [A(\mu,w)f](z)$,
is analytic for any given $w,f\in\W^1_r$. 

\begin{proof}[Proof of Lemma \ref{lem:A-operator}]
(i) For any given $w,f\in\W^1_r$ and $z\in\A_r$ the map
$\mu\mapsto w(z)f(\mu)-w(\mu)f(z)$, $\A_r\to\C$, is holomorphic, continuous on
$1/r\le|\mu|\le r$, and vanishes at $\mu=z$. This implies that the map 
\[
\mu\mapsto[A(\mu,w)f](z)=\frac{1}{z-\mu}\big(w(z)f(\mu)-w(\mu)f(z)\big),\quad\A_r\to\C,
\]
is holomorphic and continuous on $1/r\le|\mu|\le r$.
The case when $w,f\in\W^{1,+}_r$ follows in the same way.
Let us now compute the Laurent coefficients. For any $k\ge 0$ we have
\begin{align}\label{eq:A_k-derivation}
[A_k(w)f](z)&=\frac{1}{2\pi i}\oint_{|\mu|=r}
\frac{1}{z-\mu}\big(w(z)f(\mu)-w(\mu)f(z)\big)\,\frac{1}{\mu^{k+1}}\,d\mu\nonumber\\
&=f(z)\,\frac{1}{2\pi i}\oint_{|\mu|=r}\frac{1}{\mu-z}\frac{w(\mu)}{\mu^{k+1}}\,d\mu
-w(z)\,\frac{1}{2\pi i}\oint_{|\mu|=r}\frac{1}{\mu-z}\frac{f(\mu)}{\mu^{k+1}}\,d\mu\\
&=f(z)\,\Pi^+(w(z)/z^{k+1})-w(z)\,\Pi^+(f(z)/z^{k+1})\nonumber
\end{align}
where the cycles of integration are counterclockwise oriented and 
we used that\footnote{This is a direct consequence from the proof of 
the Laurent expansion theorem.}
\begin{equation}\label{eq:laurent_projection_+}
[\Pi^+g](z)=\frac{1}{2\pi i}\oint_{|\mu|=r}\frac{1}{\mu-z}g(\mu)\,d\mu
\end{equation}
for any $g\in\W^1_r$ and $z\in\A_r$.
The case when $k\le -1$ follows as in \eqref{eq:A_k-derivation} where we integrate
over the circle $|\mu|=1/r$ (oriented counterclockwise) and use that
\begin{equation}\label{eq:laurent_projection_-}
[\Pi^-_0g](z)=-\frac{1}{2\pi i}\oint_{|\mu|=1/r}\frac{1}{\mu-z}g(\mu)\,d\mu
\end{equation}
for any $g\in\W^1_r$ for $z\in\A_r$. This completes the proof of item (i).
The first statement in (ii) follows from the formulas for the Laurent coefficients
\eqref{eq:A_k}, the continuity of the projections \eqref{eq:Pi}, 
and the fact that $\W^1_r$ is a Banach algebra (Lemma \ref{lem:W-space} (ii)). 
In order to prove the second statement in (ii) we fix $k\ge 0$ and then use 
\eqref{eq:A_k} and $\|\Pi^+(g(z)/z^{k+1})\|_{W^1_r}\le\|g\|_{\W^1_r}$ 
to estimate
\begin{align}\label{eq:A_k-estimate}
\|A_k(w)f\|_{\W^1_r}&\le\big\|f\Pi^+(w(z)/z^{k+1})\big\|_{W^1_r}+
\big\|w\Pi^+(g(z)/z^{k+1})\big\|_{W^1_r}\nonumber\\
&\le C_0\big(\|f\|_{\W^1_r}\big\|\Pi^+(w(z)/z^{k+1})\big\|_{W^1_r}+
\|w\|_{\W^1_r}\big\|\Pi^+(f(z)/z^{k+1})\big\|_{W^1_r}\big)\nonumber\\
&\le C\|w\|_{\W^1_r}\|f\|_{\W^1_r},
\end{align}
where $C=2C_0>0$ is independent on the choice of $w,f\in\W^1_r$ and $k\in\Z$. 
A similar estimate shows that \eqref{eq:A_k-estimate} also holds for $k\le -1$.
This implies that $\|A_k(w)\|_{\LL(\W^1_r)}\le C\|w\|_{\W^1_r}$, and hence
the map \eqref{eq:A_k-map} is locally bounded. Since, by Lemma \ref{lem:W-space}  (ii),
for any given $f\in\W^1_r$ the map $w\mapsto A_k(w)f$, $\W^1_r\to\W^1_r$, is analytic 
(and hence, weakly analytic) we then conclude that the map \eqref{eq:A_k-map} is analytic 
(see, e.g., \cite[Theorem 3.12]{Kato}). 
This completes the proof of the Lemma in the case of the space $\W^1_r$.
The case of $\W^{1,+}_r$ then easily follows.
\end{proof}

It follows from Lemma \ref{lem:A-operator} that for any given $k\ge 0$ and 
$w\in\W^{1,+}_r$ the linear operator $\W^{1,+}_r\to\W^{1,+}_r$, $f\mapsto A_k(w)f$, 
is well-defined and bounded. (For $k\le -1$ these operators vanish.)
In addition, for any $k\ge 0$ the map
\[
\W^{1,+}_r\to\LL(\W^{1,+}_r),\quad w\mapsto A_k(w),
\]
is analytic. We have the following Proposition.

\begin{Proposition}\label{prop:A-operator}
Take $r>1$. Then we have:
\begin{itemize}
\item[(i)] For any given $\mu\in\D$ and $w\in\W^{1,+}_r$ the linear operator 
$\W^{1,+}_r\to\W^{1,+}_r$, $f\mapsto A(\mu,w)f$, where $A(\mu,w)f$ is given by the expression
\eqref{eq:A-operator}, is well-defined and bounded. Moreover, the map 
\[
\D\times\W^{1,+}_r\to\LL(\W^{1,+}_r),\quad (\mu,w)\mapsto A(\mu,w),
\]
is analytic and the series
\begin{equation}\label{eq:A-expansion(W)}
A(\mu,w)=\sum_{k\ge 0}A_k(w)\,\mu^k
\end{equation}
converges absolutely in $\LL(\W^{1,+}_r)$ uniformly on compact subsets of $\mu\in\D$
and uniformly on bounded sets of $w\in\W^{1,+}_r$.
\item[(ii)] Item (i) holds with $\W^{1,+}_r$ replaced by $\W^{1,+}_{r,0}$.
\end{itemize}
\end{Proposition}

\begin{Remark}
We note that, by construction, for any given $w,f\in\W^{1,+}_r$ and $\mu,z\in\D_r$ the expression
$[A(\mu,w)f](z)$ is well-defined and $[A(\mu,w)f](z)=\sum_{k\ge 0}[A_k(w)f](z)\,\mu^k$
where the series converges absolutely in $\C$ on compact subsets of $z\in\D_r$.
Item (i) of Proposition \ref{prop:A-operator} means that the convergence is,
in fact, stronger for $\mu$ in the unit disk $\D\subseteq\D_r$.
\end{Remark}

\begin{proof}[Proof of Proposition \ref{prop:A-operator}]
The Corollary follows from the fact that the constant $C>0$ in \eqref{eq:A_k-estimate}
is independent on the choice of $k\ge\Z$.
In fact, for any given $\mu\in\D$, $w\in\W^{1,+}_r$, and $z\in\D_r$, we have that
$[A(\mu,w)f](z)=\sum_{k\ge 0}[A_k(w)f](z)\,\mu^k$ where the series converges absolutely in $\C$. 
This and \eqref{eq:A_k-estimate} then imply that
\[
\|A(\mu,w)f\|_{\W^1_r}\le\sum_{k\ge 0}\|A_k(w)f\|_{\W^1_r}|\mu|^k\le
\sum_{k\ge 0}C\|w\|_{\W^1_r}\|f\|_{\W^1_r}|\mu|^k\le
\left(\frac{C\|w\|_{\W^1_r}}{1-|\mu|}\right)\|f\|_{\W^1_r}
\]
and hence the series \eqref{eq:A-expansion(W)} converges absolutely in $\LL(\W^{1,+}_r)$ 
uniformly on compact subsets of $\mu\in\D$ and uniformly on bounded sets of $w\in\W^{1,+}_r$.
In view of Lemma \ref{lem:A-operator} (ii), the limit operator $A(\mu,w)\in\LL(\W^1_r)$ depends
analytically on $\mu\in\D$ and $w\in\W^1_r$.
\end{proof}

\medskip

In order to define the analog of the operator
\eqref{eq:A-operator} on $\HH^1$ we introduce the left and the right shift on $\HH^1$,
\begin{equation}\label{eq:L,R}
[L(f)](y):=f(y)\,e^{-2\pi i y},\quad
[R(f)](y):=f(y)\,e^{2\pi i y}\,.
\end{equation}
It is clear that $L : \HH^1\to\HH^1$ is a linear isomorphism of Banach spaces and that
$R : \HH^1\to\HH^1$ is its inverse. In view of \eqref{eq:A_k},
for $w\in\HH^1$ and $k\in\Z$ we set
\begin{equation}\label{eq:A_k(H)}
A_k(w)f:=
\left\{
\begin{array}{l}
f\,\Pi^+\big(L^{k+1}w\big)-w\,\Pi^+\big(L^{k+1}f\big),\quad k\ge 0,\\
w\,\Pi^-_0\big(R^{|k+1|}f\big)-f\,\Pi^-_0\big(R^{|k+1|}w\big),\quad k\le -1.
\end{array}
\right.
\end{equation}
By using the identification \eqref{eq:f<->f(z)} and arguing as in the proof of
Proposition \ref{prop:A-operator} we obtain the following Proposition.

\begin{Proposition}\label{prop:A-operator(H)}
\begin{itemize}
\item [(i)] For any given $k\in\Z$ and $w\in\HH^1$ the linear operator $\HH^1\to\HH^1$,
$f\mapsto A_k(w)f$, is well-defined and continuous. Moreover, the map
$\HH^1\to\LL(\HH^1)$, $w\mapsto A_k(w)$, is analytic.
\item[(ii)] For any given $k\in\Z$ and $w\in\HH^{1,+}$ the linear operator $\HH^{1,+}\to\HH^{1,+}$,
$f\mapsto A_k(w)f$, is well-defined and continuous. The map 
$\HH^{1,+}\to\LL(\HH^{1,+})$, $w\mapsto A_k(w)$, is analytic.
\item[(iii)] For any given $\mu\in\D$ and $w\in\HH^{1,+}$ the linear operator
$\HH^{1,+}\to\HH^{1,+}$, $f\mapsto A(\mu,w)f$, where $A(\mu,w)f$ is given by
\eqref{eq:A-operator}, is well-defined and bounded.\footnote{In view of \eqref{eq:f<->f(z)}
that we identify the elements of $\HH^{1,+}$ with holomorphic functions on $\D$.} 
Moreover, the map
\[
\D\times\HH^{1,+}\to\LL(\HH^{1,+}),\quad w\mapsto A(\mu,w),
\]
is analytic and  the series 
\begin{equation}\label{eq:A-expansion(H)}
A(\mu,w)=\sum_{k\ge 0} A_k(w)\,\mu^k
\end{equation}
converges absolutely in $\LL(\HH^{1,+})$ uniformly on compact subsets of $\mu\in\D$ and 
uniformly on bounded sets of $w\in\HH^{1,+}$. In particular, the map
\[
\D\times\HH^{1,+}\to\LL(\HH^{1,+}),\quad(\mu,w)\mapsto A(\mu,w),
\]
is analytic.
\item[(iv)] Items (ii) and (iii) hold with $\HH^{1,+}$ replaced by $\HH^{1,+}_0$.
\end{itemize}
\end{Proposition}

Recall that for any given $w,f\in\W^{1,+}_r$ and $\mu\in\D_r$ the map 
$z\mapsto[A(\mu,w)f](z)$, $\D_r\to\C$, is holomorphic and continuous on 
the closure of $\D_r$. We have the following Proposition. 

\begin{Proposition}\label{prop:A-properties}
Take $r>1$. Then, the following holds:
\begin{itemize}
\item[(i)] For any given $w\in\W^{1,+}_{r,0}$ and $\mu\in\D_r$ we have that
\begin{equation}\label{eq:A(mu,w)-symmetric}
\big\1 A(\mu,w)\big(\Pi^+_0(w\check{f})\big),g\big\2=
\frac{1}{\mu}\Big([\Pi^+_0(w\check{f})](\mu)\,[\Pi^+_0(w\check{g})](\mu)
-w(\mu)[\Pi^+_0(w\check{g}\check{f})](\mu)\Big)
\end{equation}
for any $f,g\in\W^{1,+}_{r,0}$.\footnote{Note that the right side of \eqref{eq:A(mu,w)-symmetric} 
has no singularity at $\mu=0$.} In particular, we have that
\[
\big\1 A(\mu,w)\big(\Pi^+_0(w\check{f})\big),g\big\2=
\big\1 A(\mu,w)\big(\Pi^+_0(w\check{g})\big),f\big\2
\]
for any $w,f,g\in\W^{1,+}_{r,0}$ and $\mu\in\D_r$.
\item[(ii)] For any given $w\in\W^1_r$ and $\mu,\nu\in\D_r$ we have that
\[
A(\mu,w)\,A(\nu,w)f=A(\nu,w)\,A(\mu,w)f,\quad\forall f\in\W^1_r\,.
\]
\item[(iii)] Items (i) and (ii) hold with $\W^{1,+}_{r,0}$ replaced by $\HH^{1,+}_0$.
\end{itemize}
\end{Proposition}

By combining Proposition \ref{prop:A-properties} with the expansions \eqref{eq:A-expansion(W)} and
\eqref{eq:A-expansion(H)} in Proposition \ref{prop:A-operator} (i) and 
Proposition \ref{prop:A-operator(H)} (iii) we obtain the following Proposition.

\begin{Proposition}\label{prop:A-properties_k}
Take $r>1$. Then, the following holds:
\begin{itemize}
\item[(i)] For any $w,f,g\in\W^{1,+}_{r,0}$ we have that
\[
\big\1 A_k(w)\big(\Pi^+_0(w\check{f})\big),g\big\2=
\big\1 A_k(w)\big(\Pi^+_0(w\check{g})\big),f\big\2
\]
for any $k\ge 0$.
\item[(ii)] For any given $w\in\W^{1,+}_{r,0}$ and $k,l\ge 0$ the operators 
$A_k(w), A_l(w)\in\LL(\W^{1,+}_{r,0})$ commute.
\item[(iii)] Items (i) and (ii) hold with $\W^{1,+}_{r,0}$ replaced by $\HH^{1,+}_0$.
\end{itemize}
\end{Proposition}

\begin{proof}[Proof of Proposition \ref{prop:A-properties}]
(i) For any $w,f,g\in\W^{1,+}_{r,0}$ and $\mu\in\D_r$, $\mu\ne 0$, with $r>1$ we have
\begin{align*}
\big\1 A(\mu,w)\big(\Pi^+_0(w\check{f})\big),g\big\2&=
\frac{1}{2\pi i}\oint_{|z|=r}
\frac{w(z)[\Pi^+_0(w\check{f})](\mu)-w(\mu)[\Pi^+_0(w\check{f})](z)}{z-\mu}\,
\frac{\check{g}(z)}{z}\,dz\\
&=[\Pi^+_0(w\check{f})](\mu)\,\Pi^+\big(w(\mu)\check{g}(\mu)/\mu\big)
-w(\mu)\Pi^+\big(\check{g}(\mu)\,[\Pi^+_0(w\check{f})](\mu)/\mu\big)\\
&=[\Pi^+_0(w\check{f})](\mu)\,\Pi^+\big(w(\mu)\check{g}(\mu)/\mu\big)
-w(\mu)\Pi^+\big(w(\mu)\check{f}(\mu)\check{g}(\mu)/\mu\big)\\
&=\frac{1}{\mu}\Big([\Pi^+_0(w\check{f})](\mu)\,[\Pi^+_0(w\check{g})](\mu)
-w(\mu)[\Pi^+_0(w\check{f}\check{g})](\mu)\Big)
\end{align*}
where we used \eqref{eq:laurent_projection_+}, the fact that
\[
\Pi^+\big(\check{g}(\mu)\,[\Pi^+_0(w\check{f})](\mu)/\mu\big)
=\Pi^+\big(\check{g}(\mu)\,\big[w\check{f}-\Pi^-\big(w\check{f})\big](\mu)/\mu\big)
=\Pi^+\big(w(\mu)\check{f}(\mu)\check{g}(\mu)/\mu\big),
\]
and the identity $\Pi^+(f(\mu)/\mu)=[\Pi^+_0 f](\mu)/\mu$.
Since by Proposition \ref{prop:A-operator} (i), the both sides of \eqref{eq:A(mu,w)-symmetric}
are holomorphic in $\mu\in\D$, the equality \eqref{eq:A(mu,w)-symmetric} holds also for $\mu=0$.
This proves item (i). Item (ii) is proved by direct computation.
The case of the space $\HH^{1,+}_0$ follows by identifying the elements 
of $\HH^{1,+}_0$ with holomorphic functions on $\D$ and then arguing as 
in the case of $\W^{1,+}_{r,0}$.
\end{proof}

Let us also record the property
\begin{equation*}
A_k(w)\,f=-A_k(f)\,w
\quad\text{and}\quad A(\mu,w)\,f=-A(\mu,f)\,w
\end{equation*}
for any $\mu\in\D$, $w,f\in\HH^{1,+}$, and $k\ge 0$.
This property follows from \eqref{eq:A_k(H)} and \eqref{eq:A-expansion(H)}.
For any given $k\ge 1$ and $w\in\HH^{1,+}_0$ we will also need the adjoint operator to
$A_k(w) : \HH^{1,+}_0\to\HH^{1,+}_0$ with respect to the pairing \eqref{eq:L^2-metric(H0)},
defined by the property that
\begin{equation}\label{eq:A_k-adjoint}
\1 A_k(w) f,g\2=\1 f, A_k(w)^T g\2
\end{equation}
for any $f,g\in\HH^{1,+}_0$.

\section{The Poisson structure and the phase space}\label{sec:phase_space}
In this Section we define a Poisson bracket on the phase space $\HH^{1,+}_0\times\HH^{1,+}_0$
and discuss several basic properties of the Hamiltonian function \eqref{eq:H-introduction} defined
in the Introduction. Consider the Sobolev space $\HH^1$ and the corresponding Hardy space
\[
\HH^{1,+}_0=\big\{f\in\HH^1\,\big|\,f_k=0\,\,\text{for}\,\,k\le 0\big\}\subseteq\HH^1
\]
equipped with the restriction of the scalar product \eqref{eq:H1-norm}.
The restriction of the bilinear form \eqref{eq:L2-pairing(H)} to 
$\HH^{1,+}_0$ equals
\begin{equation}\label{eq:L^2-metric(H0)}
\1 w,p\2=\sum_{k\ge 1} w_k p_k=\int_0^1 w(y)\check{p}(y)\,dy,\quad w,p\in\HH^{1,+}_0,
\end{equation}
where $\check{u}(y)=u(-y)$, $y\in\T$. In what follows we will use the same notation
for the various restrictions of linear operators and forms to linear subspaces.

\medskip

\noindent{\em The Poisson bracket.}
For any $C^1$ functions $F,G : \HH^{1,+}_0\times\HH^{1,+}_0\to\C$ consider the following
Zakharov-Shabat-type Poisson bracket
\begin{equation}\label{eq:Poisson_bracket}
\{F,G\}:=dG(X_F):=\1\nabla_p F,\nabla_w G\2-\1\nabla_w F,\nabla_p G\2
\end{equation}
where $dG$ is the differential of $G$ and the gradients $\nabla_w G,\nabla_p G\in\mathcal{D}^+_0$ are 
uniquely defined (by Riesz representation theorem) so that
\[
\1\nabla_w G,\delta w\2=\big(D_1 G\big)(\delta w)\quad\text{\rm and}
\quad \1\nabla_p G,\delta p\2=\big(D_2 G\big)(\delta p)
\]
hold for any $\delta w, \delta p\in\HH^{1,+}_0$, where $D_1G$ and $D_2G$ denote
the directional partial derivatives of $G$ in the direction of the first, respectively, 
the second variable, and $\1\cdot,\cdot\2$ is the bilinear form \eqref{eq:L^2-metric(H0)}.
For any $C^1$ function $F : \HH^{1,+}_0\times\HH^{1,+}_0\to\C$ consider the Hamiltonian ``vector field''
\begin{equation}\label{eq:X_F}
X_F:=\big(\nabla_p F,-\nabla_w F\big)\,.
\end{equation}
Note that, in general, the two components of \eqref{eq:X_F} are elements of $\mathcal{D}^+_0$, 
and hence, the bracket \eqref{eq:Poisson_bracket} is not always defined.

As in the Introduction, consider the Hamiltonian 
\begin{equation}\label{eq:Hamiltonian}
H(w,p)=\1 w,p^2\2=\int_0^1\check{w}(y)\,p(y)^2\,dy,\quad w,p\in\HH^{1,+}_0 .
\end{equation}
Since $\HH^{1,+}_0$ is a Banach algebra, the Hamiltonian function 
$H : \HH^{1,+}_0\times\HH^{1,+}_0\to\C$ is well-defined and analytic. 
Let us derive the Hamiltonian equation of \eqref{eq:Hamiltonian}. 
For any two variations $\delta w,\delta p\in\HH^{1,+}_0$ we obtain from
\eqref{eq:Hamiltonian} and \eqref{eq:check-move} that
\[
\big(D_2H\big)(\delta p)=\1 w,2 p\,\delta p\2=2\1 w\check{p},\delta p\2
\quad\text{and}\quad
\big(D_1H\big)(\delta w)=\1 p^2,\delta w\2\,.
\]
This implies that 
\begin{equation}\label{eq:H-gradients}
\nabla_w H=p^2
\quad\text{\rm and}\quad
\nabla_p H=2\Pi^+_0(w\check{p})
\end{equation}
where $\Pi^+_0 : \HH^1\to\HH^{1,+}_0$ is the continuous projection onto $\HH^{1,+}_0$
(see \eqref{eq:Pi(H)}). In particular, we see that the Hamiltonian equation 
corresponding to \eqref{eq:Hamiltonian} is
\begin{equation}\label{eq:X_H}
\dot{w}=\nabla_p H=2\Pi^+_0(w\check{p}),\quad\dot{p}=-\nabla_w H=-p^2,
\end{equation}
where the dot denotes the time derivative. It then follows from \eqref{eq:X_H} that
the map (see \eqref{eq:X_F})
\[
X_H : \HH^{1,+}_0\times\HH^{1,+}_0\to\HH^{1,+}_0\times\HH^{1,+}_0
\]
is well-defined and analytic. Hence, \eqref{eq:X_H} is an ordinary differential equation (ODE) 
in $\HH^{1,+}_0$. One easily sees that \eqref{eq:X_H} can be explicitly solved.

\begin{Remark}\label{rem:mechanics}
Borrowing terminology from the classical mechanics, we will think of 
the Hamiltonian \eqref{eq:Hamiltonian} as the kinetic energy of a moving
point with configuration space $\HH^{1,+}_0$ and a phase space $\HH^{1,+}_0\times\HH^{1,+}_0$ 
equipped with the Poisson bracket \eqref{eq:Poisson_bracket}. 
The phase space is coordinatized by the variables $(w,p)\in\HH^{1,+}_0\times\HH^{1,+}_0$ where
$w\in\HH^{1,+}_0$ represents the position of the point and $p\in\HH^{1,+}_0$ is its impulse. 
The equation \eqref{eq:X_H} is the geodesics equation of $H$ in $\HH^{1,+}_0\times\HH^{1,+}_0$. 
We will see later that $H$ can be interpreted as an infinite dimensional analog of 
a Riemannian metric on $\R^n$ with a completely integrable geodesic flow.
\end{Remark}

In view of \eqref{eq:check-move} the Hamiltonian \eqref{eq:Hamiltonian} 
can be written as
\begin{equation}\label{eq:H_Hankel_form}
H(w,p)=\1 \Gamma(w)\,p,p\2
\end{equation}
where $w,p\in\HH^{1,+}_0$ and $\Gamma(w)$ is the Hankel operator \eqref{eq:Hankel-introduction}.
The arguments in the proof of Proposition \ref{prop:A-operator(H)}
show that the map
\begin{equation}\label{eq:w->H(w)-analytic}
\HH^{1,+}_0\to\LL(\HH^{1,+}_0),\quad w\mapsto\Gamma(w),
\end{equation}
is analytic. In view of Remark \ref{rem:mechanics}, the operator 
\eqref{eq:Hankel-introduction} is the ``inertia tensor'' of the Hamiltonian \eqref{eq:Hamiltonian}.
It follows easily from the Banach algebra property of $\HH^{1,+}_0$ that
the operator \eqref{eq:Hankel-introduction} is compact for any given $w\in\HH^{1,+}_0$.
In view of \eqref{eq:check-move},
\begin{equation}\label{eq:H(w)-symmetric}
\1\Gamma(w)f,g\2=\1 w\check{f},g\2=\1 w,fg\2=\1\Gamma(w)g,f\2
\end{equation}
for any $w,f,g\in\HH^{1,+}_0$. It follows from Proposition \ref{prop:A-properties} and
Proposition \ref{prop:A-properties_k} that
\begin{equation}\label{eq:A(w)H(w)-symmetric}
\1 A(\mu,w)\Gamma(w) f,g\2=\1 A(\mu,w)\Gamma(w) g,f\2\quad\text{and}\quad
\1 A_k(w)\Gamma(w) f,g\2=\1 A_k(w)\Gamma(w) g,f\2
\end{equation}
for any $w,f,g\in\HH^{1,+}_0$, $\mu\in\D$, and $k\ge 0$. 

\begin{Remark}\label{rem:Legendre_transform}
The map 
\begin{equation}\label{eq:Legendre_transform}
\HH^{1,+}_0\times\HH^{1,+}_0\to\HH^{1,+}_0\times\HH^{1,+}_0,
\quad (w,p)\mapsto\Big(w,\big(\nabla_p H\big)(w,p)/2\Big),
\end{equation}
where $\big(\nabla_p H\big)(w,p)=2\Gamma(w)p$ is an analog of the Legendre transform corresponding 
to the Hamiltonian $\frac{1}{2}H(w,p)$. It transforms impulses into velocities. 
Note that \eqref{eq:Legendre_transform} is {\em not} invertible, since the operator 
\eqref{eq:Hankel-introduction} is compact.
\end{Remark}

\begin{Remark}
Note that formula \eqref{eq:Poisson_bracket} defines a (weak) Poisson bracket of $C^1$ functions
on the phase space $\W^{1,+}_{r,0}\times\W^{1,+}_{r,0}$, $r>0$. It follows from the properties of 
the space $\W^{1,+}_{r,0}$ that the definitions and the statements proved in this Section continue 
to hold with $\HH^{1,+}_0$ replaced by $\W^{1,+}_{r,0}$, $r>1$.
\end{Remark}

\section{The integrals and the momentum map}\label{sec:integrals_f=0}
In this Section we derive explicit formulas for the expressions \eqref{eq:I_k-introduction}
considered in the Introduction. For any given $w\in\HH^{1,+}_0$ and $k\ge 0$ we have that
\begin{equation}\label{eq:I_k}
I_k(w,p)=\1 A_k(w)\Gamma(w) p,p\2
=\big\1 A_k(w)\big(\Pi^+_0(w\check{p})\big),p\big\2
\end{equation}
where $A_k(w) : \HH^{1,+}_0\to\HH^{1,+}_0$ is given by \eqref{eq:A_k(H)}. 
It follows from Proposition \ref{prop:A-operator(H)} (iv)
that $I_k : \HH^{1,+}_0\times\HH^{1,+}_0\to\C$ is well-defined and analytic.
In addition to \eqref{eq:I_k}, for any given $w\in\HH^{1,+}_0$ consider the $1$-parameter family of
quadratic in $p\in\HH^{1,+}_0$ function
\begin{equation}\label{eq:I(mu)}
I(\mu;w,p):=\1 A(\mu,w)\Gamma(w) p,p\2
=\big\1 A(\mu,w)\big(\Pi^+_0(w\check{p})\big),p\big\2
=\sum_{k\ge 0}I_k(w,p)\,\mu^k
\end{equation}
where $\mu\in\D$ is a parameter and $A(\mu,w) : \HH^{1,+}_0\to\HH^{1,+}_0$
is the operator \eqref{eq:A-expansion(H)} (see Proposition \ref{prop:A-operator(H)} (iv)).
By Proposition \ref{prop:A-operator(H)} (iv) and the analyticity of \eqref{eq:w->H(w)-analytic}, 
for any given $\mu\in\D$ the function 
$I(\mu) : \HH^{1,+}_0\times\HH^{1,+}_0\to\C$, $(w,p)\mapsto I(\mu;w,p)$, 
is well-defined and analytic. The series in \eqref{eq:I(mu)} converges uniformly on
compact subsets of $\mu\in\D$ and uniformly on bounded sets of $\HH^{1,+}_0\times\HH^{1,+}_0$.

In order to formulate the results, we introduce the left and the right shifts 
$S_\pm : \HH^{1,+}_0\to \HH^{1,+}_0$ on $\HH^{1,+}_0$,
\begin{equation}\label{eq:S_pm}
S_-f:=\Pi^+_0(L f),\quad S_+f:=R f,
\end{equation}
where $L$ and $R$ are the left an the right shifts on $\HH^1$ (see \eqref{eq:L,R}).
The operators $S_\pm : \HH^{1,+}_0\to\HH^{1,+}_0$ are bounded and satisfy
\begin{equation}\label{eq:S_pm-properties}
\1 S_-u,v\2=\1 u,S_+v\2,\quad S_-S_+u=u,\quad S_+S_-u=u-u_1 e_1,\quad\text{and}\quad
S_+(uv)=uS_+v,
\end{equation}
for any $u,v\in\HH^{1,+}_0$, where we set
\begin{equation}\label{eq:e_l}
e_l(y):=e^{2\pi i l y},\quad y\in\T,\quad l\in\Z.
\end{equation}
Clearly, $\{e_l\,|\,l\in\Z\}$ is a basis in $\HH^1$.
We have the following explicit formulas for \eqref{eq:I_k} and \eqref{eq:I(mu)}.

\begin{Proposition}\label{prop:I-formulas}
\begin{itemize}
\item[(i)] For any $w,p\in\HH^{1,+}_0$ we have that
\begin{equation}\label{eq:I_k-formula}
I_k(w,p)=\sum_{1\le j\le k}\1 S_-^jw,p\2 \1 S_-^kw,S_-^{j-1}p\2-
w_j\1 S_-^kw,p\,S_-^{j-1}p\2\big),\quad k\ge 1,
\end{equation}
and $I_0(w,p)=0$ for $k=0$.
For any given $k\ge 1$ the map $I_k : \HH^{1,+}_0\times\HH^{1,+}_0\to\C$ is
analytic.
\item[(ii)] For any $w,p\in\HH^{1,+}_0$ and $\mu\in\D$, we have that
\begin{align}\label{eq:I(mu)-formula}
I(\mu;w,p)=\frac{1}{\mu}\left([\Gamma(w)p]^2(\mu)-w(\mu)\,[\Gamma(w)(p^2)](\mu)\right)
\end{align}
where $\Gamma(w) : \HH^{1,+}_0\to\HH^{1,+}_0$ is the Hankel operator \eqref{eq:Hankel-introduction} and
the right side is holomorphic for $\mu\in\D$ and continuous on its closure.
Moreover, for any given $w,p\in\HH^{1,+}_0$ the expression \eqref{eq:I(mu)-formula}, 
when considered as a function $y\mapsto I(e^{2\pi i y};w,p)$, $\T\to\C$,
belongs to $\HH^{1,+}_0$ and the momentum map
\begin{equation}\label{eq:I-momentum_map(H)}
I : \HH^{1,+}_0\times\HH^{1,+}_0\to\HH^{1,+}_0,\quad (w,p)\mapsto 
I(w,p):=S_-\left([\Gamma(w)p]^2-w\,\Gamma(w)(p^2)\right),
\end{equation}
is analytic.
\end{itemize}
\end{Proposition}

\noindent As mentioned in the Introduction, the functions \eqref{eq:I_k} are Poisson commuting
integrals of the Hamiltonian flow of $H$ on $\HH^{1,+}_0\times\HH^{1,+}_0$ 
(respectively $\W^{1,+}_{r,0}\times\W^{1,+}_{r,0}$). This explains our choice of terminology
for the map \eqref{eq:I-momentum_map(H)}.

\begin{proof}[Proof of Proposition \ref{prop:I-formulas}]
(i) By \eqref{eq:A_k(H)}, for any $w,f\in\HH^{1,+}_0$ we have that
$A_0(w)f=0$ and, for $k\ge 1$,
\begin{align}
A_k(w)f&=f\,\big(L^{k+1}w-\Pi^-_0\big(L^{k+1}w\big)\big)
-w\,\big(L^{k+1}f-\Pi^-_0\big(L^{k+1}f\big)\big)\nonumber\\
&=w\,\Pi^-_0\big(L^{k+1}f\big)-f\,\Pi^-_0\big(L^{k+1}w\big)\label{eq:A_k-shifts_bis}\\
&=\big(\sum_{l\ge k+1}w_le_l\big)\big(\sum_{1\le j\le k}f_j e_j\big)\,e_{-(k+1)}
-\big(\sum_{l\ge k+1}f_le_l\big)\big(\sum_{1\le j\le k}w_je_j\big)\,e_{-(k+1)}\nonumber\\
&=\sum_{1\le j\le k}\big(f_jS_+^{j-1}S_-^kw-w_jS_+^{j-1}S_-^kf\big).\label{eq:A_k-shifts}
\end{align}
We then obtain from \eqref{eq:I_k} and \eqref{eq:A_k-shifts} that for $k\ge 1$,
\begin{align*}
I_k(w,p)&=\big\1 A_k(w)\big(\Pi^+_0(w\check{p})\big),p\big\2\nonumber\\
&=\sum_{1\le j\le k}(w\check{p})_j \1 S_+^{j-1}S_-^kw,p\2-
w_j\big\1 S_+^{j-1}S_-^k\Pi^+_0(w\check{p}),p\big\2\big)\nonumber\\
&=\sum_{1\le j\le k}\1 w, S_+^jp\2 \1 S_-^kw,S_-^{j-1}p\2-
w_j\1 S_-^kw,p\,S_-^{j-1}p\2\big).
\end{align*}
where we used that $(w\check{p})_j=\1 w\check{p},e_j\2=\1 w,S_+^jp\2$ as well as
\eqref{eq:S_pm-properties} and \eqref{eq:check-move} to conclude that
$\big\1 S_+^{j-1}S_-^k\Pi^+_0(w\check{p}),p\big\2=\1 S_-^kw,p\,S_-^{j-1}p\2$.
The last statement in (i) was proved earlier.
This completes the prove of item (i).

Let us now prove item (ii). 
In fact, it follows \eqref{eq:I(mu)} and Proposition \ref{prop:A-properties} that
for any $w,p\in\HH^{1,+}_0$ and $\mu\in\D$ we have 
\begin{align}\label{eq:I(mu)-formula1}
I(\mu;w,p)=\big\1 A(\mu,w)\big(\Pi^+_0(w\check{p})\big),p\big\2=
\frac{1}{\mu}\left([\Pi^+_0(w\check{p})]^2(\mu)-w(\mu)\,[\Pi^+_0(w\check{p}^2)](\mu)\right).
\end{align}
The analyticity of the map $I : \HH^{1,+}_0\times\HH^{1,+}_0\to\HH^{1,+}_0$
follows since $\HH^{1,+}_0$ is a Banach algebra.
This completes the proof of the Proposition.
\end{proof}

The Corollary below follows directly from \eqref{eq:X_F},  \eqref{eq:I_k-formula}, 
and the fact that $\HH^{1,+}_0$ is a Banach algebra (see the proof of \eqref{eq:H-gradients}).

\begin{Corollary}\label{coro:X_I_k-vector-field(H)}
For any given $k\ge 1$ the map
\[
X_{I_k} : \HH^{1,+}_0\times\HH^{1,+}_0\to\HH^{1,+}_0\times\HH^{1,+}_0
\]
is well-defined and analytic. 
\end{Corollary}

\begin{Remark}
Note that the definitions and the statements proved in this Section continue 
to hold with $\HH^{1,+}_0$ replaced by $\W^{1,+}_{r,0}$, $r>1$.
\end{Remark}

\section{Involutivity of the integrals}\label{sec:involutivity_f=0}
In this Section we prove Theorem \ref{th:involutivity} stated in the Introduction.
We start with a preparation. For a given $\mu\in\D$ consider the delta function 
$\delta_\mu : \HH^{r,+}_0\to\C$, $g\mapsto g(\mu)$.
Since $\HH^{1,+}_0$ is boundedly embedded in the space $C(\D)$ of continuous complex-valued
functions on $\D$, $\delta_\mu$ is a bounded  linear functional on $\HH^{1,+}_0$, 
$\delta_\mu\in\big(\HH^{1,+}_0\big)^*$. 
Moreover, we have that
\begin{equation}\label{eq:tilde_delta1}
\delta_\mu g=\1\tilde{\delta}_\mu,g\2,\quad\forall g\in\HH^{1,+}_0,
\end{equation}
where the holomorphic function $\tilde{\delta}_\mu : \D\to\C$,
\begin{equation}\label{eq:tilde_delta2}
\tilde{\delta}_\mu(z):=\sum_{k\ge 1}\mu^k z^k=\frac{\mu z}{1-\mu z},
\end{equation}
has a finite $\HH^1$-norm
$\|\tilde{\delta}_\mu\|_{\HH^1}=\big(\sum_{k\ge 1}|\mu|^{2k} k^2\big)^{1/2}<\infty$,
and hence
\[
\tilde{\delta}_\mu\in\HH^{1,+}_0\quad\text{for}\quad \mu\in\D.
\]
Let us now fix $\mu\in\D$ and consider the analytic function
\begin{equation}\label{eq:I(mu)-map}
I(\mu) : \HH^{1,+}_0\times\HH^{1,+}_0\to\C,\quad
(w,p)\mapsto \delta_\mu[I(w,p)]=I(\mu;w,p),
\end{equation}
where $I(w,p)$ is the map \eqref{eq:I-momentum_map(H)}.
The Lemma below follows directly from \eqref{eq:X_F} and 
\eqref{eq:I(mu)-formula}.

\begin{Lemma}\label{lem:X_I(mu)}
For any given $\mu\in\D$ we have that 
$X_{I(\mu)}=\big(\nabla_p I(\mu),-\nabla_w I(\mu)\big)$ where
\[
\nabla_w I(\mu)=\alpha_\mu\,p\,\tilde{\delta}_\mu-\beta_\mu\,\tilde{\delta}_\mu-
\gamma_\mu\,p^2\,\tilde{\delta}_\mu,\quad
\nabla_p I(\mu)=\alpha_\mu\,\Pi^+_0(w\,\check{\tilde{\delta}}_\mu)-
2\gamma_\mu\,\Pi^+_0(w\,\check{p}\,\check{\tilde{\delta}}_\mu),
\]
$\tilde{\delta}_\mu\in\HH^{1,+}_0$ is the holomorphic function \eqref{eq:tilde_delta2},
and $\alpha_\mu:=2[\Pi^+_0(w\check{p})](\mu)/\mu$, $\beta_\mu:=[\Pi^+_0(w\check{p}^2)](\mu)/\mu$,
and $\gamma_\mu:=w(\mu)/\mu$.
\end{Lemma}

\begin{proof}[Proof of Lemma \ref{lem:X_I(mu)}]
For any given $w,p\in\HH^{1,+}_0$ and $\mu\in\D$ we obtain 
from \eqref{eq:I(mu)-formula} and \eqref{eq:tilde_delta1} that
\[
I(\mu;w,p)=
\frac{1}{\mu}\Big(\Big\1[\Pi^+_0(w\check{p})]^2,\widetilde{\delta}_\mu\Big\2
-\1 w,\widetilde{\delta}_\mu\,\2\big\1\Pi^+_0(w\check{p}^2),\widetilde{\delta}_\mu\big\2\Big)\,.
\]
This implies that for any $\delta p\in\HH^{1,+}_0$ we have that
\begin{align*}
\big(D_2 I(\mu)\big)(w,p)(\delta p&)=
\frac{1}{\mu}\Big(
2[\Gamma(w)p](\mu)\,\big\1\Gamma(w)\delta p,\widetilde{\delta}_\mu\big\2
-w(\mu)\,\big\1 2\Pi^+_0\big(w\check{p}(\delta p)^\vee\big),\widetilde{\delta}_\mu\big\2\Big)\\
&=\frac{1}{\mu}\Big(
2[\Gamma(w)p](\mu)\,\big\1\Gamma(w)\widetilde{\delta}_\mu,\delta p\big\2
-2 w(\mu)\,\big\1\Gamma(w)(p\widetilde{\delta}_\mu),\delta p\big\2
\Big)\\
&=\alpha_\mu\big\1\Gamma(w)\widetilde{\delta}_\mu,\delta p\big\2
-2\gamma_\mu\big\1\Gamma(w)(p\widetilde{\delta}_\mu),\delta p\big\2
\end{align*}
where we used \eqref{eq:H(w)-symmetric} and the fact that
\[
\big\1\Pi^+_0\big(w\check{p}(\delta p)^\vee\big),\widetilde{\delta}_\mu\big\2
=\big\1 w\check{p}(\delta p)^\vee,\widetilde{\delta}_\mu\big\2
=\big\1 \check{w}p(\delta p),\widecheck{\widetilde{\delta}}_\mu\big\2
=\big\1\delta p,\Pi^+_0(w\check{p}\widecheck{\widetilde{\delta}}_\mu)\big\2
=\big\1\delta p,\Gamma(w)(p\widetilde{\delta}_\mu)\big\2\,.
\]
This implies the formula for the gradient $\nabla_pI(\mu)$ in Lemma \ref{lem:X_I(mu)}.
The formula for $\nabla_wI(\mu)$ can be proved in the same way.
\end{proof}

\noindent Note that $\alpha_\mu$, $\beta_\mu$, and $\gamma_\mu$ do not have a singularity 
at $\mu=0$. Since $\HH^{1,+}_0$ is a Banach algebra and since the projection 
$\Pi^+_0 : \HH^1_r\to\HH^{1,+}_0$ is continuous, we have the following version
of Corollary \ref{coro:X_I_k-vector-field(H)}.

\begin{Corollary}\label{coro:X_I(mu)-vector_field(H)}
For any given $\mu\in\D$ the map
\[
X_{I(\mu)} : \HH^{1,+}_0\times\HH^{1,+}_0\to\HH^{1,+}_0\times\HH^{1,+}_0
\]
is well-defined and analytic. 
\end{Corollary}

It follows from \eqref{eq:L2-pairing(W)}, \eqref{eq:tilde_delta1}, 
and \eqref{eq:tilde_delta2}, that for any given $\mu,\nu\in\D$, $\mu\ne\nu$, and 
$g\in\HH^1$ we have
\begin{align}\label{eq:cycle1}
\1 g,\tilde{\delta}_\mu\,\tilde{\delta}_\nu\2=\1 g-g_0,\tilde{\delta}_\mu\,\tilde{\delta}_\nu\2
=\frac{1}{2\pi i}\oint_{|z|=1}\frac{\mu\nu\,\big(g(z)-g_0\big)}{(z-\mu)(z-\nu)}\,\frac{dz}{z}
=\frac{\mu g(\nu)}{\nu-\mu}+\frac{\nu g(\mu)}{\mu-\nu}
\end{align}
where we used the residue formula and the fact that $\big(g(z)-g_0\big)/z$ is holomorphic 
for $z\in\D$.

\begin{proof}[Proof of Theorem \ref{th:involutivity}]
The Theorem follows easily from Lemma \ref{lem:X_I(mu)}. In fact, take 
$\mu,\nu\in D$, $\mu\ne\nu$.
It then follows from Lemma \ref{lem:X_I(mu)}, \eqref{eq:check-move}, and \eqref{eq:cycle1}, 
that for any $w,p\in\HH^{1,+}_0$ we have
\begin{align}\label{eq:cycle2}
\big\1\nabla_p I(\mu),\nabla_w I(\nu)\big\2&=
\alpha_\mu\alpha_\nu\1\check{\tilde{\delta}}_\mu w,\tilde{\delta}_\nu p\2
-\alpha_\mu\beta_\nu\1\check{\tilde{\delta}}_\mu w,\tilde{\delta}_\nu\2
-\alpha_\mu\gamma_\nu\1\check{\tilde{\delta}}_\mu w,\tilde{\delta}_\nu p^2\2\nonumber\\
&-2\gamma_\mu\alpha_\nu\1\check{\tilde{\delta}}_\mu w\check{p},\tilde{\delta}_\nu p\2
+2\gamma_\mu\beta_\nu\1\check{\tilde{\delta}}_\mu w\check{p},\tilde{\delta}_\nu\2
+2\gamma_\mu\gamma_\nu\1\check{\tilde{\delta}}_\mu w\check{p},\tilde{\delta}_\nu p^2\2\nonumber\\
&=\Big(\alpha_\mu\alpha_\nu\1 w\check{p},\tilde{\delta}_\mu\,\tilde{\delta}_\nu\2
+2\gamma_\mu\gamma_\nu\1 w\check{p}^3,\tilde{\delta}_\mu\,\tilde{\delta}_\nu\2\Big)\nonumber\\
&-\alpha_\mu\beta_\nu\Big(\gamma_\nu\frac{\mu\nu}{\nu-\mu}+\gamma_\mu\frac{\nu\mu}{\mu-\nu}\Big)
-\alpha_\mu\gamma_\nu
\Big(\beta_\nu\frac{\mu\nu}{\nu-\mu}+\beta_\mu\frac{\nu\mu}{\mu-\nu}\Big)\nonumber\\
&-2\gamma_\mu\alpha_\nu
\Big(\beta_\nu\frac{\mu\nu}{\nu-\mu}+\beta_\mu\frac{\nu\mu}{\mu-\nu}\Big)
+\gamma_\mu\beta_\nu
\Big(\alpha_\nu\frac{\mu\nu}{\nu-\mu}+\alpha_\mu\frac{\nu\mu}{\mu-\nu}\Big)\nonumber\\
&=\Big(\alpha_\mu\alpha_\nu\1 w\check{p},\tilde{\delta}_\mu\,\tilde{\delta}_\nu\2
+2\gamma_\mu\gamma_\nu\1 w\check{p}^3,\tilde{\delta}_\mu\,\tilde{\delta}_\nu\2\Big)\nonumber\\
&+\frac{2\mu\nu}{\nu-\mu}\big(\alpha_\nu\beta_\mu\gamma_\mu-\alpha_\mu\beta_\nu\gamma_\nu\big)
+\frac{\mu\nu}{\nu-\mu}\big(\alpha_\mu\beta_\mu\gamma_\nu-\alpha_\nu\beta_\nu\gamma_\mu\Big)
\end{align}
where we also used the definition of $\alpha_\mu$, $\beta_\mu$, and $\gamma_\mu$, $\mu\in\D$.
Since the right side in \eqref{eq:cycle2} is symmetric with respect to the variables
$\mu$ and $\nu$, we conclude from \eqref{eq:Poisson_bracket} that
\[
\{I(\mu),I(\nu)\}=
\big\1\nabla_p I(\mu),\nabla_w I(\nu)\big\2-\big\1\nabla_p I(\nu),\nabla_w I(\mu)\big\2=0
\quad\forall\mu,\nu\in\D.
\]
(The case when $\mu=\nu$ is trivial.)
Similarly, it follows from Lemma \ref{lem:X_I(mu)} and \eqref{eq:H-gradients} that
for any given $\mu\in\D$ and $w,p\in\HH^{1,+}_0$ we have
\begin{align*}
\{I(\mu),H\}&=\alpha_\mu\1\check{\tilde{\delta}}_\mu w,p^2\2
-2\gamma_\mu\1\check{\tilde{\delta}}_\mu w \check{p},p^2\2
-2\alpha_\mu\1 w\check{p},\tilde{\delta}_\mu p\2
+2\beta_\mu\1 w\check{p},\tilde{\delta}_\mu\2
+2\gamma_\mu\1 w\check{p},\tilde{\delta}_\mu p^2\2\\
&=-\alpha_\mu\1\tilde{\delta}_\mu,w\check{p}^2\2
+2\beta_\mu\1\tilde{\delta}_\mu,w\check{p}\2=\mu\alpha_\mu\beta_\mu+\mu\beta_\mu\alpha_\mu=0
\end{align*}
where we used that 
$\1\tilde{\delta}_\mu,w\check{p}^2\2=[\Pi^+_0(w\check{p}^2)](\mu)=\mu\beta_\mu$
and 
$2\1\tilde{\delta}_\mu,w\check{p}\2=2[\Pi^+_0(w\check{p})](\mu)=\mu\alpha_\mu$.
It follows from and \eqref{eq:I(mu)} and Proposition \ref{prop:A-operator(H)} (iv) that 
the series
\[
[I(\mu)](w,p)=I(\mu;w,p)=\sum_{k\ge 1} I_k(w,p)\,\mu^k
\]
converges uniformly on compact subsets of $\mu\in\D$ and uniformly on bounded sets of
$w,p\in\HH^{1,+}_0\times\HH^{1,+}_0$. This implies that for any given 
$w,p\in\HH^{1,+}_0\times\HH^{1,+}_0$ the series
\[
\{I(\mu),I(\nu)\}=\sum_{k,l\ge 1}\{I_k,I_l\}\,\mu^k\nu^l\quad\text{and}\quad
\{I(\mu),H\}=\sum_{k\ge 1}\{I_k,H\}\,\mu^k
\]
converge absolutely and uniformly on compact sets of $\mu,\nu\in\D$.
This implies that
\[
\{I_k,I_l\}=0\quad\text{and}\quad\{I_k,H\}=0
\]
for any $k,l\ge 1$.
This completes the proof of Theorem \ref{th:involutivity}.
\end{proof}

\begin{Remark}
In the case of the space $\W^{1,+}_{r,0}$, $r>1$, we take $\mu\in\D_{1/r}$ and note
that for any $g\in\W^{1,+}_{r,0}$ the delta function $\delta_\mu : \W^{1,+}_{r,0}\to\C$ satisfies 
$\delta_\mu g=\1 g,\widetilde\delta_\mu\2$ where $\widetilde\delta_\mu$ is given by 
\eqref{eq:tilde_delta2} and $\widetilde\delta_\mu\in\W^{1,+}_{r,0}$. With this modification, 
the statements proved in this Section continue to hold with $\HH^{1,+}_0$ replaced 
by $\W^{1,+}_{r,0}$.
\end{Remark}

\section{Systems with potential}\label{sec:potentials}
In this Section we consider the case of potentials and prove Theorem \ref{th:J_k-involutivity}.
As in the Introduction, for any given $f\in\HH^1$ and $w\in\mathcal{O}$ consider the potential
\begin{equation}\label{eq:U^f}
V_f(w)=\1 1,f/w\2=\int_0^1\frac{f(y)}{w(y)}\,dy
\end{equation}
where $\mathcal{O}=\big\{w\in\HH^{1,+}_0\,\big|\,w(y)\ne 0\,\forall y\in\T\big\}$ is an open
and dense set in $\HH^{1,+}_0$. The map
\begin{equation}\label{eq:U^f-map}
V_f : \mathcal{O}\to\C,\quad w\mapsto V_f(w),
\end{equation}
is well-defined and analytic by the following Lemma.

\begin{Lemma}\label{lem:1/w-analytic}
The map $\OO\to\HH^1$, $w\mapsto 1/w$, is well-defined and analytic.
In particular, the map \eqref{eq:U^f-map} is analytic for any choice of $f\in\HH^1$.
\end{Lemma}

\begin{proof}[Proof of Lemma \ref{lem:1/w-analytic}]
The Lemma follows from Lemma \ref{lem:1/w} (iii) and the Banach algebra property 
of $\HH^1$. In fact, since $\OO$ is open in $\HH^1$ and since $1/w\in\HH^1$ for $w\in\OO$,
for any given $w_0\in\OO$ we can choose $\kappa>0$ such that for any $g\in\HH^1$ with 
$\|g\|_{\HH^1}<\kappa$ we have that $w_0+g\in\OO$ and
$\kappa<1/(2C\|1/w_0\|_{\HH^1})$ where $C>0$ is the Banach algebra constant
in $\HH^1$. Then, for any $g\in\HH^1$ with $\|g\|_{\HH^1}<\kappa$ we have that
\[
\frac{1}{w_0+g}=\frac{1}{w_0}\left(\frac{1}{1+(1/w_0)g}\right)=
(1/w_0)\sum_{k=0}^\infty(-1)^k[(1/w_0)g]^k
\]
where the series converges in $\HH^1$ uniformly in $\|g\|_{\HH^1}<\kappa$ since
\[
\|(1/w_0)g\|_{\HH^1}\le C\|1/w_0\|_{\HH^1}\|g\|_{\HH^1}<1/2.
\]
This completes the proof of the first statement of the Lemma.
The second statement follows from the Banach algebra property of $\HH^1$.
\end{proof}

It follows from \eqref{eq:U^f} that 
$(D_1 V_f)(\delta w)=-\1 1,(f/w^2)\delta w\2=-\1\check{f}/\check{w}^2,\delta w\2$
for any $w\in\OO$ and $\delta w\in\HH^{1,+}_0$. This implies that
\begin{equation}\label{eq:U^f-gradient}
\nabla_w V_f=-\Pi^+_0\big(\check{f}/\check{w}^2\big)
\end{equation}
for any $w\in\OO$. It then follows from \eqref{eq:U^f-gradient} and 
Lemma \ref{lem:1/w-analytic} that $\nabla_w V_f\in\HH^{1,+}_0$ and that the map
\begin{equation}\label{eq:U^f-gradient_map}
\OO\to\HH^{1,+}_0,\quad w\mapsto\nabla_w V_f,
\end{equation}
is analytic.
As in the Introduction, we now fix $f\in\HH^1$ and consider the Hamiltonian
\begin{equation}\label{eq:H_f}
H_f(w,p)=H(w,p)+V_f(w),\quad (w,p)\in\OO\times\HH^{1,+}_0\,.
\end{equation}
It follows from Lemma \ref{lem:1/w-analytic} and the analyticity of $H$ that
the function $H_f : \mathcal{O}\times\HH^{1,+}_0\to\C$ is analytic. 
We will now prove that the Hamiltonian flow of $H_f$ admits an infinite family of 
integrals in involution on $\mathcal{O}\times\HH^{1,+}_0$.
The proof is based on the following Lemma. As in the Introduction, for any $k\ge 1$ and 
$w\in\OO$ consider the individual potential
\begin{equation}\label{eq:U_k}
U_k^f(w)=-\big(w\,\Pi^+_0(f/w)\big)_{k+1}=-\sum_{1\le j\le k}w_j\big\1 f/w,e_{k+1-j}\big\2\,.
\end{equation}
It follows from from \eqref{eq:U_k} and Lemma \ref{lem:1/w-analytic} that the map,
\begin{equation}\label{eq:U_k-map} 
\OO\to\C,\quad w\mapsto U_k^f(w),
\end{equation}
is analytic. We have that
\begin{equation}\label{eq:gradient U_k}
\big\1\nabla U_k^f(w),\delta w\big\2=\sum_{1\le j\le k}\big\1(f/w^2)(\delta w),e_{k+1-j}\big\2\,w_j
-\big\1 f/w,e_{k+1-j}\big\2\,\delta w_j,
\end{equation}
and hence, 
$\nabla_w U_k^f(w)=
\sum_{1\le j\le k}w_j\Pi^+_0\big((\check{f}/\check{w}^2)\,e_{k+1-j}\big)
-\big\1 f/w,e_{k+1-j}\big\2\,e_j$ .
This implies that $\nabla_w U_k^f$ belongs to $\HH^{1,+}_0$ and the map
\begin{equation}\label{eq:U_k-gradient_map}
\OO\to\HH^{1,+}_0,\quad w\mapsto\nabla_w U_k^f,
\end{equation}
is analytic for any $k\ge 1$.
Recall that $A_k(w)^T$ denotes the adjoint operator to $A_k(w)$ with
respect to the pairing \eqref{eq:L^2-metric(H0)} in $\HH^{1,+}_0$
(cf. \eqref{eq:A_k-adjoint}).

\begin{Lemma}\label{lem:U_k}
For any given $f\in\HH^1$ and $k\ge 1$ and the function 
$U_k^f : \OO\to\C$ is analytic and
\begin{equation}\label{eq:U_k-property}
A_k(w)^T(\nabla_w V_f)=\nabla_w U_k^f,\quad w\in\OO,
\end{equation}
where $(a)_k$ is the $k$-th Fourier efficient of $a\in\HH^1$.
\end{Lemma}

\begin{proof}[Proof of Lemma \ref{lem:U_k}]
Take $f\in\HH^1$ and $k\ge 1$.
For any $w\in\mathcal{O}$ and $\delta w\in\HH^{1,+}_0$ we then 
obtain from \eqref{eq:A_k-shifts_bis} and \eqref{eq:U^f-gradient} that
\begin{align*}
\big\1 A_k(w)^T(\nabla_w V_f),\delta w\big\2&=\big\1\nabla_w V_f,A_k(w)\delta w\2
=-\big\1w\,\Pi^-_0\big(L^{k+1}\delta w\big)-\delta w\,\Pi^-_0\big(L^{k+1} w\big),
\check{f}/\check{w}^2\big\2\nonumber\\
&=\sum_{1\le j\le k}
\big\1(\delta w)\,w_j\,e_{j-k-1},\check{f}/\check{w}^2\big\2
-\big\1 w\,(\delta w_j)\,e_{j-k-1},\check{f}/\check{w}^2\big\2\\
&=\sum_{1\le j\le k}
\big\1(f/w^2)(\delta w),e_{k+1-j}\big\2\,w_j
-\big\1 f/w,e_{k+1-j}\big\2\,\delta w_j\\
&=\1\nabla_w U_k^f,\delta w\2
\end{align*}
where we used \eqref{eq:gradient U_k} to conclude the last equality.
This completes the proof of the Lemma.
\end{proof}

\noindent As in the Introduction, for any given $k\ge 1$ and 
$(w,p)\in\mathcal{O}\times\HH^{1,+}_0$ consider the expression
\begin{equation}\label{eq:J_k}
J_k^f(w,p)=I_k(w,p)+U_k^f(w)\,.
\end{equation}
It follows from Proposition \ref{prop:I-formulas} and the analyticity of \eqref{eq:U_k-map}
that the map $J_k^f : \mathcal{O}\times\HH^{1,+}_0\to\C$ is 
analytic for any $k\ge 1$. By summarizing the above, be obtain the following analog of 
Corollary \ref{coro:X_I_k-vector-field(H)}.

\begin{Corollary}\label{coro:X_J_k-vector-field(H)}
Take $f\in\HH^1$. Then we have:
\begin{itemize}
\item[(i)] The functions $H_f, J_k^f : \OO\times\HH^{1,+}_0\to\C$ are
analytic for any $k\ge 1$.
\item[(ii)] The maps 
$X_{H_f}, X_{J_k^f} : \OO\times\HH^{1,+}_0\to\HH^{1,+}_0\times\HH^{1,+}_0$
are well-defined and analytic. 
\end{itemize}
\end{Corollary}

\begin{proof}[Proof of Corollary \ref{coro:X_J_k-vector-field(H)}]
Item (i) is proved above.
Let us prove (ii). For any given $k\ge 1$ and $(w,p)\in\OO\times\HH^{1,+}_0$
we obtain from \eqref{eq:H_f}, \eqref{eq:J_k}, and  \eqref{eq:H-gradients}, that
\begin{align}\label{eq:X_H_f}
X_{H_f}(w,p)&=\big(\nabla_p H_f,-\nabla_w H_f\big)
=\big(2\Gamma(w)p,-p^2-\nabla_w V_f\big)
\end{align}
\begin{align}\label{eq:X_J_k}
X_{J_k^f}(w,p)&=\big(\nabla_p J_k^f,-\nabla_w J_k^f\big)
=\big(\nabla_p I_k,-\nabla_w I_k-\nabla_w U_k^f\big)\,.
\end{align}
Item (ii) then follows from \eqref{eq:X_H_f}, \eqref{eq:X_J_k},
Corollary \ref{coro:X_I_k-vector-field(H)},
and the analyticity of the maps \eqref{eq:U^f-gradient_map}
and \eqref{eq:U_k-gradient_map}.
\end{proof}
 
\noindent Let us now prove Theorem \ref{th:J_k-involutivity} stated in the Introduction.

\begin{proof}[Proof of Theorem \ref{th:J_k-involutivity}]
The Theorem follows from Theorem \ref{th:involutivity} and Lemma \ref{lem:U_k}.
In fact, take $f\in\HH^1$. Then, for any $k\ge 1$ and 
$(w,p)\in\OO\times\HH^{1,+}_0$ we obtain from Theorem \ref{th:involutivity}
and Lemma \ref{lem:U_k} that
\begin{align*}
\{H_f,J_k^f\}&=\{H_f+V_f,I_k+U_k^f\}=\{H_f,U_k^f\}-\{I_k,V_f\}
=\1\nabla_p H_f,\nabla_w U_k^f\2-\1\nabla_p I_k,\nabla_w V_f\2\\
&=2\1\Gamma(w)p,A_k(w)^T\nabla_w V_f\2-2\1 A_k(w)\Gamma(w)p,\nabla_w V_f\2=0
\end{align*}
where we used \eqref{eq:A_k-adjoint} and the fact that, by 
\eqref{eq:A(w)H(w)-symmetric} and \eqref{eq:I_k},
\begin{equation}\label{eq:D_pI_k}
\big(\nabla_p I_k\big)(w,p)=2 A_k(w)\Gamma(w)p
\end{equation}
for any $k\ge 1$ and $(w,p)\in\OO\times\HH^{1,+}_0$.
Similarly, for any $k,l\ge 1$ and $(w,p)\in\OO\times\HH^{1,+}_0$ we have that
\begin{align*}
\{J_k^f,J_l^f\}&=\{I_k+U_k^f,I_l+U_l^f\}=\{I_k,U_l^f\}-\{I_l,U_k^f\}
=\1\nabla_p I_k,\nabla_w U_l^f\2-\1\nabla_p I_l,\nabla_w U_k^f\2\\
&=2\1A_k(w)\Gamma(w)p,A_l(w)^T\nabla_w V^f\2-2\1 A_l(w)\Gamma(w)p,A_k(w)^T\nabla_w V^f\2=0
\end{align*}
in view of the commutativity of the operators $A_k(w)$ and $A_l(w)$ 
(cf. Proposition \ref{prop:A-properties_k}).
This completes the proof of the Theorem.
\end{proof}

Let us now turn to the momentum map of the integrals $J^f_k$, $k\ge 1$.
It follows from \eqref{eq:I(mu)}, \eqref{eq:U_k}, \eqref{eq:J_k}, and 
Proposition \ref{prop:I-formulas} (ii), that for any 
$(w,p)\in\OO\times\HH^{1,+}_0$ and $\mu\in\D$ we have
\begin{align}\label{eq:J(mu)}
J^f(\mu;w,p)&:=\sum_{k\ge 1}J_k^f(w,p)\,\mu^k
=\sum_{k\ge 1}\big(I_k(\mu;w,p)+U_k^f(w)\big)\,\mu^k\nonumber\\
&=\1 A(\mu,w)\Gamma(w) p,p\2+[U^f(w)](\mu)/\mu
\end{align}
where we set
\[
[U^f(w)](\mu):=-[w\,\Pi^+_0(f/w)](\mu)\,,
\]
and the series converge absolutely.
(In fact, since $w\,\Pi^+_0(f/w)\in\HH^{1,+}_0$,
the convergence is uniform on the closure of $\D$.)
In particular, by \eqref{eq:I(mu)-formula}, we have that for $\mu\in\D$,
\begin{align}\label{eq:J(mu)-formula}
J^f(\mu;w,p)=\frac{1}{\mu}\left([\Pi^+_0(w\check{p})]^2-w\,\Pi^+_0(w\check{p}^2)
-w\,\Pi^+_0(f/w)\right)(\mu)\,.
\end{align}
In the same way as in Proposition \ref{prop:I-formulas} (ii),
the right side in \eqref{eq:J(mu)-formula} defines a map
\begin{equation}\label{eq:J-momentum_map(H)}
J^f : \OO\times\HH^{1,+}_0\to\HH^{1,+}_0,\quad 
(w,p)\mapsto J^f(w,p),
\end{equation}
where 
\begin{equation}\label{eq:J(w,p)}
J^f(w,p)=S_-\left([\Pi^+_0(w\check{p})]^2-w\,\Pi^+_0(w\check{p}^2)
-w\,\Pi^+_0(f/w)\right)\,.
\end{equation}
This is exactly the momentum map \eqref{eq:J^f-introduction} considered in the Introduction.
By construction, the components of \eqref{eq:J(w,p)} are the integrals
$J_k^f(w,p)$, $k\ge 1$.
In view of \eqref{eq:J(w,p)} and the Banach algebra property of $\HH^{1,+}_0$
we obtain Proposition \ref{prop:J^f-introduction} stated in the Introduction.
In the same way as in \eqref{eq:I(mu)-map}, for any given $\mu\in\D$, 
we obtain an analytic function
\begin{equation}\label{eq:J(mu)-map}
J^f(\mu) : \OO\times\HH^{1,+}_0\to\C,\quad (w,p)\mapsto\delta_\mu[J^f(w,p)],
\end{equation}
where $\delta_\mu : \HH^{1,+}_0\to\C$, $g\mapsto g(\mu)$, is the delta function.
As in Section \ref{sec:involutivity_f=0}, for any given $\mu\in\D$
the function $\tilde{\delta}_\mu : \D\to\C$ given by \eqref{eq:tilde_delta2} belongs to 
$\HH^{1,+}_0$ and satisfies
\begin{equation}\label{eq:tilde_delta3}
\delta_\mu g=\1 \tilde{\delta}_\mu,g\2,\quad\forall g\in\HH^{1,+}_0\,.
\end{equation}
By arguing in the same way as in the proof of Lemma \ref{lem:X_I(mu)} we then
obtain from \eqref{eq:J(mu)-map} and \eqref{eq:tilde_delta3} the following
analog of Corollary \ref{coro:X_I(mu)-vector_field(H)}.

\begin{Corollary}\label{coro:X_J(mu)-vector_field(H)}
For any given $\mu\in\D$ the map
\[
X_{J^f(\mu)} : \OO\times\HH^{1,+}_0\to\HH^{1,+}_0\times\HH^{1,+}_0
\]
is well-defined and analytic. 
\end{Corollary}

\noindent Note that it follows from \eqref{eq:A(w)H(w)-symmetric} and \eqref{eq:J(mu)} 
that for any given $\mu\in\D$ we have that
\begin{equation}\label{eq:nabla_pJ(mu)}
\nabla_p J^f(\mu)=2 A(\mu,w)\Gamma(w)p
\end{equation}
for any $(w,p)\in\OO\times\HH^{1,+}_0$.
Since $U^f(w)\in\HH^{1,+}_0$ for $w\in\OO$ we have that
$U^f(w)=\sum_{k\ge 1} U_k^f(w)\,e_k$ where $e_k$, $k\ge 1$, 
is the standard basis in $\HH^{1,+}_0$ and the series converges in $\HH^{1,+}_0$. 
In view of the analyticity of the map $\HH^{1,+}_0\to\HH^{1,+}_0$, $w\mapsto U^f(w)$, 
we then conclude that
$\1\nabla_w U^f,\delta w\2=\sum_{k\ge 1}\1\nabla_w U_k^f,\delta w\2\,e_k$
for any given $w\in\OO$ and $\delta w\in\HH^{1,+}_0$ where 
the series converges in $\HH^{1,+}_0$. In particular, we obtain that
$\1\nabla_w [U^f(w)](\mu),\delta w\2=\sum_{k\ge 1} \1\nabla_w U_k^f,\delta w\2\,\mu^k$
for any given $w\in\OO$ and $\delta w\in\HH^{1,+}_0$ where the series converges 
uniformly in $\mu\in\D$.
This implies that for any given $w\in\OO$ we have that
\[
\{H_f,J^f(\mu)\}=\sum_{k\ge 1}\{H_f,J_l^f\}\,\mu^k\,
\quad\text{and}\quad
\{J^f(\mu),J^f(\nu)\}=\sum_{k\ge 1}\sum_{l\ge 1}\{J_k^f,J_l^f\}\,\mu^k\nu^l
\]
for any $\mu,\nu\in\D$. By combining this with Theorem \ref{th:J_k-involutivity} we obtain 
the following addition to Theorem \ref{th:J_k-involutivity}.
 
\begin{Theorem}\label{th:J(mu)-involutivity}
Take $f\in\HH^1$. Then we have:
\begin{itemize}
\item[(i)] For any given $\mu\in\D$ the analytic function
\eqref{eq:J(mu)-map} is an integral of the Hamiltonian flow of $H_f$.
\item[(ii)] $\{J^f(\mu),J^f(\nu)\}=\{J^f(\mu),J_k^f\}=0$ for any $\mu,\nu\in\D$ and $k\ge 1$.
\end{itemize}
\end{Theorem}

\begin{Remark}
Note that it follows from Lemma \ref{lem:U_k}, Proposition \ref{prop:A-operator(H)} (ii),
and the expansion in $\mu\in\D$ for the gradient $\nabla_w U^f$ above, that
\begin{equation*}
A(\mu,w)^T\big(\nabla_w V_f\big)=\nabla_w U^f(\mu)
\end{equation*}
for any given $\mu\in\D$  and $w\in\OO$. 
\end{Remark}

It follows from Corollary \ref{coro:X_J_k-vector-field(H)},
Corollary \ref{coro:X_J(mu)-vector_field(H)}, and the theorem on the existence and
uniqueness of solutions of ODEs in Banach spaces (see, e.g., \cite[Ch. IV, \S 1]{Lang})
that the Hamiltonian flows with Hamiltonians $H_f$, $J^f_k$, $k\ge 1$, and $J^f(\mu)$
where $f\in\HH^1$ and $\mu\in\D$ are locally uniquely defined on $\OO\times\HH^{1,+}_0$.
More specifically, we have the following Theorem.

\begin{Theorem}\label{th:existence}
For any given $(w_\bullet,p_\bullet)\in\OO\times\HH^{1,+}_0$ there exists 
an open neighborhood $\mathcal{V}$ of $(w_\bullet,p_\bullet)$ in $\OO\times\HH^{1,+}_0$ and $T>0$ 
such that for any initial data $(w_0,p_0)\in\mathcal{V}$ there exists a unique solution 
\begin{equation}\label{eq:(w,p)-solution}
(w,p)\in C^\ell\big([0,T),\OO\times\HH^{1,+}_0\big),\quad\forall\ell\ge 1,
\end{equation}
of the Hamiltonian equation with Hamiltonian $H_f$ such that $(w,p)|_{t=0}=(w_0,p_0)$.
The solution depends continuously on the initial data $(w_0,p_0)$ in the sense that
the data-to-solution map 
\[
\OO\times\HH^{1,+}_0\to C^\ell\big([0,T),\OO\times\HH^{1,+}_0\big),\quad
(w_0,p_0)\mapsto(w,p),
\]
is continuous.
The same holds with $H_f$ replaced by $J^f_k$, $k\ge 2$, and $J^f(\mu)$.
\end{Theorem}

\begin{Remark}\label{rem:existence}
Note that the time of existence $T>0$ in Theorem \ref{th:existence} is independent of
the choice of the regularity exponent $\ell\ge 1$.
In fact, since the corresponding Hamiltonian vector fields are analytic, 
the open neighborhood $\mathcal{V}$ and the time of existence $T>0$ can be chosen so that 
the solution \eqref{eq:(w,p)-solution} extends to an analytic map $\D_T\to\OO\times\HH^{1,+}_0$
(cf., e.g., \cite{Goursat}).
\end{Remark}

\begin{Remark}
In the case of the space $\W^{1,+}_{r,0}$, $r>1$, we consider the open set 
$\OO:=\big\{w\in\W^{1,+}_{r,0}\,\big|\,w(\mu)\ne 0\,\,\text{\rm for }\,1/r\le|\mu|\le r\big\}$
in $\W^{1,+}_{r,0}$. Then, by Proposition \ref{lem:1/w} (ii), for any choice of $f\in\W^1_r$ 
the potential $V_f(w):=\1 1,f/w\2$ is well-defined for any $w\in\OO$ and the map 
$V_f :\OO\to\C$, $w\mapsto V_f(w)$, is analytic. With this modification, 
the statements proved in this Section continue  to hold with $\HH^{1,+}_0$ replaced 
by $\W^{1,+}_{r,0}$.
\end{Remark}

\section{Finite dimensional dynamics and geodesic equivalence}\label{sec:finite_dynamics}
In this Section we consider a class of initial data $(w_0,p_0)\in\OO\times\HH^{1,+}_0$ 
(cf. Theorem \ref{th:existence}) that lead to a finite dimensional dynamics of 
the Hamiltonian system with Hamiltonian $H_f$, $f\in\HH^1$, and show that it is
related to the Hamiltonian flow of a (pseudo) Riemannian metric (with potential) 
of a special type.

\noindent{\em The reduced system}. 
Take $N\ge 2$ and consider the affine plane $\mathcal{F}_N$ in 
the configuration space $\HH^{1,+}_0$,
\[
\mathcal{F}_N:=\big\{w=w_1 e_1+...+w_N e_N+e_{N+1}\,\big|\,w_k\in\C,\,k=1,...,N\big\},
\]
as well as the submanifolds
\[
\mathcal{PF}_N:=(\OO\cap\mathcal{F}_N)\times
\big\{p=p_1 e_1+...+p_N e_N\,\big|\,\,p_k\in\C,\,k=1,...,N\big\}
\]
and
\[
\mathcal{HF}_N:=(\OO\cap\mathcal{F}_N)\times\HH^{1,+}_0
\]
in the phase space $\OO\times\HH^{1,+}_0$. 
Here $e_k$, $k\in\Z$, denotes denotes the basis \eqref{eq:e_l} in $\HH^1$.
Clearly, we have that $\mathcal{PF}_N\subseteq\mathcal{HF}_N$. We will also need the projection
\begin{equation}\label{eq:pi_N}
\pi_N : \mathcal{HF}_N\to\mathcal{PF}_N,\quad
(w,p)\mapsto(w,\Pi^+_{\le N}p),
\end{equation}
where $\Pi^+_{\le N} : \HH^{1,+}_0\to\HH^{1,+}_0$, $p\mapsto \sum_{k=1}^N p_k e_k$,
is the projection onto the first $N$ components of $\HH^{1,+}_0$.
We equip the $2N$ (complex) dimensional manifold $\mathcal{PF}_N$ with 
the canonical Poisson structure
\begin{equation}\label{eq:canonical_poisson_N}
\{w_k,w_l\}_N=\{p_k,p_l\}_N=0,\quad\{p_l,w_k\}_N=\delta_{lk},
\quad 1\le k,l\le N\,.
\end{equation}
We will also need the restriction of the Hamiltonian 
\begin{equation}\label{eq:H_f-tilde}
\widetilde{H}_f:=H_f\big|_{\mathcal{PF}_N}
\end{equation}
and the integrals 
\begin{equation}\label{eq:J_f-tilde}
\widetilde{J}_k^f:=J_k^f\big|_{\mathcal{PF}_N},\quad 1\le k\le N,
\end{equation}
to the submanifold $\mathcal{PF}_N$ (cf. \eqref{eq:J_k} for the definitions of the integrals).
One has the following Lemma.

\begin{Lemma}\label{lem:finite_dynamics}
Take $f\in\HH^1$ and $N\ge 2$. Then we have:
\begin{itemize}
\item[(i)] The submanifold $\mathcal{HF}_N$ is invariant with respect to the Hamiltonian flow
of $H_f$ and the integrals $J_k^f$, $k\ge 1$. Equivalently, we have that
$\{H_f,w_l\}|_{(w,p)}=\{J_k^f,w_l\}|_{(w,p)}=0$ for any $(w,p)\in\mathcal{HF}_N$, $k\ge 1$, 
and $l\ge N+1$.
\item[(ii)] For any $(w,p)\in\mathcal{HF}_N$ we have that
\[
(d_{(w,p)}\pi_N)\big(X_{H_f}\big)=X_{\widetilde{H}_f}\big(\pi_N(w,p)\big)
\quad\text{and}\quad
(d_{(w,p)}\pi_N)\big(X_{J_k^f}\big)=X_{\widetilde{J}_k^f}\big(\pi_N(w,p)\big)
\]
for any $k\ge 1$ where $\pi_N$ is the projection \eqref{eq:pi_N} and
$X_{\widetilde{H}_f}$ and $X_{\widetilde{J}_k^f}$ are the Hamiltonian vector fields on 
$\mathcal{PF}_N$ with respect to the canonical Poisson structure \eqref{eq:canonical_poisson_N}.
\item[(iii)] For any $(w,p)\in\mathcal{HF}_N$ we have that
\[
\{H_f,J_k^f\}|_{(w,p)}=\{\widetilde{H}_f,\widetilde{J}_k^f\}_N(\pi_N(w,p))
\quad\text{and}\quad
\{J_k^f,J_l^f\}|_{(w,p)}=\{\widetilde{J}_k^f,\widetilde{J}_l^f\}_N(\pi_N(w,p))
\]
for any $1\le k,l\le N$. In particular, the Hamiltonian $\widetilde{H}_f$ and
the functions $\widetilde{J}_k^f$, $1\le k\le N$, are in involution with respect to the 
Poisson structure \eqref{eq:canonical_poisson_N} on $\mathcal{PF}_N$.
\end{itemize}
\end{Lemma}

Lemma \ref{lem:finite_dynamics} implies the following characterization of
the Hamiltonian flow of $H_f$ with initial data on $\mathcal{HF}_N$.

\begin{Proposition}\label{prop:finite_dynamics}
Take $f\in\HH^1$, $N\ge 2$, and let $(w,p)\in C^\ell\big([0,T),\OO\times\HH^{1,+}_0\big)$,
$\ell\ge 2$, $T>0$, be the solution \eqref{eq:(w,p)-solution} of the Hamiltonian equation with 
Hamiltonian $H_f$ and initial data $(w_0,p_0)\in\mathcal{HF}_N$ (cf. Theorem \ref{th:existence}).
Then, the $w$-component of the solution belongs to $\mathcal{F}_N$,
$w\in C^\ell\big([0,T),\mathcal{F}_N\big)$, and the curve
\[
(w,\Pi^+_{\le N}p)\in C^\ell\big([0,T),\mathcal{PF}_N\big)
\]
is a solution of the Hamiltonian equation on $\mathcal{PF}_N$ with Hamiltonian 
$\widetilde{H}_f$ and initial data $(w_0,\Pi^+_{\le N}p_0)$. The flow of $\widetilde{H}_f$
on $\mathcal{PF}_N$ has $N$ integrals in involution.
These statements also hold with $H_f$ replaced by $J_k^f$, $1\le k\le N$.
\end{Proposition}

\begin{Remark}
Proposition \ref{prop:finite_dynamics} suggests the following procedure for obtaining 
the solution of $X_{H_f}$ with initial data $(w_0,p_0)\in\mathcal{HF}_N$: We first obtain
the solution of $X_{\widetilde{H}_f}$ on $\mathcal{PF}_N$ with initial data $(w_0,\Pi^+_{\le N}p_0)$. 
We then substitute the $w$-component $t\mapsto w(t)$, $[0,T)\to\mathcal{F}_N$, 
of the obtained solution into $X_{H_f}$ and obtain the ordinary differential equation 
in $\HH^{1,+}_0$,
\[
\dot{p}=-\big(\nabla_w H_f\big)(w(t),p),\quad p|_{t=0}=p_0,
\]
on the $p$-component of the solution of $X_{H_f}$.
In this sense, we can say that the solutions of $X_{H_f}$
exhibit a finite dimensional dynamics on the submanifold $\mathcal{HF}_N$.
The above also holds with $H_f$ replaced by $J_k^f$,$1\le k\le N$.
Lemma \ref{lem:finite_dynamics} (and Proposition \ref{prop:finite_dynamics}) can be 
considered as an instance of a Hamiltonian reduction in infinite dimensions.
\end{Remark}

\begin{proof}[Proof of Lemma \ref{lem:finite_dynamics}]
We will prove the Lemma in the case of the integrals. The case of the Hamiltonian $H_f$
follows in the same way. Take $f\in\HH^1$ and $N\ge 2$ and recall from \eqref{eq:I_k-formula}, 
\eqref{eq:U_k}, and \eqref{eq:J_k}, that for $(w,p)\in\OO\times\HH^{1,+}_0$ and $k\ge 1$,
\[
J_k^f(w,p)=I_k(w,p)+U_k^f(w)
\]
where
\[
I_k(w,p)=\sum_{1\le j\le k}\1 S_-^j w, p\2 \1 S_+^{j-1}S_-^k w,p\2-
w_j\1 S_-^kw,p\,S_-^{j-1}p\2\big)\quad\text{and}\quad
U_k^f(w)=-\big\1 w\Pi^+_0(f/w),e_{k+1}\big\2\,.
\]
Now, assume that $w\in\OO\cap\mathcal{F}_N$ and choose $k\ge 1$. 
Then, one sees from the expressions above 
that $J_k^f(w,p)$ does not depend on the variables $p_j$, $j\ge N+1$, or equivalently,
\begin{equation}\label{eq:J_k-fiber}
J_k^f(w,p)=\widetilde{J}^f_k(\pi_N(w,p))\quad\text{for}\quad(w,p)\in\mathcal{HF}_N.
\end{equation}
It follows easily from the definition of the Poisson structure \eqref{eq:Poisson_bracket} that
the coordinates $w_l, p_l, l\ge 1$, on $\HH^{1,+}_0\times\HH^{1,+}_0$ are canonical,
\[
\{w_m,w_l\}=\{p_m,p_l\}=0,\quad\{p_l,w_m\}=\delta_{lm},\quad m,l\ge 1.
\]
Denote by $\partial_{w_l}$ and $\partial_{p_l}$ the partial derivatives with
respect to the coordinates $w_l, p_l, l\ge 1$, on $\HH^{1,+}_0\times\HH^{1,+}_0$.
We have that
\begin{equation}\label{eq:J_k-fiber_bis}
L_{X_{J_k^f}} w_l=\{J_k^f,w_l\}=\partial_{p_l} J_k^f=0,\quad l\ge N+1,
\end{equation}
where we used \eqref{eq:J_k-fiber} to conclude that the partial derivative of $J_k^f$ with 
respect to the variable $p_l$ vanishes and where $L_X$ denotes the directional derivative
in the direction of a vector field $X$.
This proves item (i). Similarly, for $1\le l\le N$,
\begin{equation}\label{eq:J_k-fiber_ter}
L_{X_{J_k^f}} w_l=\{J_k^f,w_l\}=
\partial_{p_l} J_k^f,\quad L_{X_{J_k^f}} p_l=\{J_k^f,p_l\}=
-\partial_{w_l}J_k^f,
\end{equation}
and
\begin{equation}\label{eq:J_k-fiber_quater}
L_{X_{\widetilde{J}_k^f}} w_l=\{\widetilde{J}_k^f,w_l\}_N=\partial_{p_l}\widetilde{J}_k^f
=\partial_{p_l} J_k^f,\quad
L_{X_{\widetilde{J}_k^f}} p_l=\{\widetilde{J}_k^f,p_l\}_N=-\partial_{w_l}\widetilde{J}_k^f
=-\partial_{w_l} J_k^f\,.
\end{equation}
Items (ii) and (iii) now follow from \eqref{eq:J_k-fiber_bis}, \eqref{eq:J_k-fiber_ter},
\eqref{eq:J_k-fiber_quater}, and \eqref{eq:J_k-fiber}.
\end{proof}

\medskip

\noindent{\em The structure of the momentum map on $\mathcal{HF}_N$.}
Let us now study the structure of the momentum map \eqref{eq:J(w,p)} on $\mathcal{HF}_N$ and
its dependence on the choice of $f\in\HH^1$. Take $f\in\HH^1$ and $(w,p)\in\OO\times\HH^{1,+}_0$. 
Then, by \eqref{eq:I-momentum_map(H)} and \eqref{eq:J(w,p)}, we have that
\begin{equation}\label{eq:J(w,p)_bis}
J^f(w,p)=I(w,p)-S_-\left(w\,\Pi^+_0(f/w)\right)
\end{equation}
where $I(w,p)=S_-\left([\Pi^+_0(w\check{p})]^2-w\,\Pi^+_0(w\check{p}^2)\right)$
are the integrals in the potential-free case (cf. Section \ref{sec:integrals_f=0}).
We rewrite the term depending on $f\in\HH^1$ in \eqref{eq:J(w,p)_bis} as follows
\begin{align}
w\,\Pi^+_0(f/w)&=w\,\big(f/w-\Pi^-(f/w)\big)=f-w\,\Pi^-(f/w)\nonumber\\
&=\big(\Pi^+_0 f-\Pi^+_0\big(w\,\Pi^-(f/w)\big)\big)+
\big(\Pi^- f-\Pi^-\big(w\,\Pi^-(f/w)\big)\big)\nonumber\\
&=\Pi^+_0 f-\Pi^+_0\big(w\,\Pi^-(f/w)\big)\,.\label{eq:equality_integrals}
\end{align}
The term $\Pi^- f-\Pi^-\big(w\,\Pi^-(f/w)\big)$ vanishes since
$w\,\Pi^+_0(f/w)$ and $\Pi^+_0 f-\Pi^+_0\big(w\,\Pi^-(f/w)\big)$ belong to $\HH^{1,+}_0$.
By combining this with \eqref{eq:J(w,p)_bis} we conclude that
\begin{equation}\label{eq:J(w,p)_ter}
J^f(w,p)=I(w,p)+S_-\left(\Pi^+_0\big(w\,\Pi^-(f/w)\big)-\Pi^+_0f\right)\,.
\end{equation}
For later use, we also note that the term $w\,\Pi^+_0(f/w)$, and hence,
by \eqref{eq:equality_integrals}, the term $\Pi^+_0 f-\Pi^+_0\big(w\,\Pi^-(f/w)\big)$, 
are spanned on the vectors $e_k$, $k\ge 2$. We have the following Lemma.

\begin{Lemma}\label{lem:J(w,p)-decomposition}
Assume that $(w,p)\in\mathcal{HF}_N$. Then for any choice of $f\in\HH^1$ we have that
\[
J^f(w,p)=\sum_{k=1}^N J_k^f(w,p)\,e_k+\sum_{k\ge N+1} f_{k+1}\,e_k
\]
where the integrals $J_k^f(w,p)$, $1\le k\le N$, are defined in \eqref{eq:J_k}.
If in addition, the roots of the polynomial $w(z)=w_1z+...+w_N z^N+z^{N+1}$ are contained in 
the unit disk $\D$ then the integrals $J_k^f(w,p)$, $1\le k\le N$, do not
depend on $f_k$ with $k\le N+1$.
\end{Lemma}

\begin{proof}[Proof of Lemma \ref{lem:J(w,p)-decomposition}]
The first statement of the Lemma follows from \eqref{eq:J(w,p)_ter}.
In fact, since $w\in\mathcal{F}_N$, the term $w\,\Pi^-(f/w)$ is spanned
on the vectors $e_k$, $k\le N+1$. This implies that
$S_-\left(\Pi^+_0\big(w\,\Pi^-(f/w)\big)\right)$
is spanned on $e_k$, $0\le k\le N$, and hence,
\[
S_-\Big(\Pi^+_0\big(w\,\Pi^-(f/w)\big)-\sum_{k=1}^{N+1}f_k\,e_k\Big)\in
\text{\rm Span}_\C\big(e_1,...e_N\big)\,.
\]
On the other side, since $S_-^kw=0$ for $w\in\mathcal{F}_N$ and $k\ge N+1$,
we obtain from \eqref{eq:I_k-formula} that $I_k(w,p)=0$ for any $k\ge N+1$,
$w\in\mathcal{F}_N$, and $p\in\HH^{1,+}_0$.
By combining this with \eqref{eq:J(w,p)_ter} we conclude the proof
of the first statement of the Lemma.
Let us now prove the second statement.
Assume that the roots of the polynomial $w(z)=w_1 z+...+w_N z^N+z^{N+1}$ are 
contained in the unit disk $\D$. Then,
\begin{equation}\label{eq:w-polynomial}
w(z)=z(z-z_1)\cdots(z-z_N),\quad z_1,...,z_N\in\D,
\end{equation}
and hence, for $z$ on the unit circle $z\in\mathbb{S}$,
\begin{equation}\label{eq:f/w}
\frac{f(z)}{w(z)}=\frac{f(z)}{z^{N+1}}\,\prod_{j=1}^N\frac{1}{1-z_j/z}
=\frac{f(z)}{z^{N+1}}\,\prod_{j=1}^N\left(\sum_{k\ge 0}\Big(z_j/z\Big)^k\right)
=\frac{f(z)}{z^{N+1}}\,g(z)
\end{equation}
where, in view of the identification \eqref{eq:f<->f(1/z)}, \eqref{eq:f(z)}, 
$g(z)$ is an element of $\HH^{1,-}$.
It now follows from \eqref{eq:f/w} and the fact that $g\in\HH^{1,-}$ that
\[
\Pi^+_0(f/w)=0\quad\text{for}\quad f=\sum_{k\le N+1} f_k\,e_k\,.
\]
By combining this with \eqref{eq:J(w,p)_bis} we then conclude the proof
of the second statement of the Lemma.
\end{proof}

\medskip

\noindent{\em The range of the $Q$-transform on $\mathcal{HF}_N$.}
Here we discuss the range of the map
\begin{equation}\label{eq:Q_f}
\OO\to\HH^{1,+},\quad w\mapsto Q_f(w):=\Pi^+(f/w^2),
\end{equation}
in the case when it is restricted to the submanifold $\OO\cap\mathcal{F}_N$, $N\ge 2$.
(The map \eqref{eq:Q_f} is related to the $Q$-transform \eqref{eq:Q+} studied in 
Section \ref{sec:PDEs}.)
Take $f\in\HH^1$ and assume that $w\in\OO\cap\mathcal{F}_N$ so that the zeros of 
the polynomial $w(z)=w_1 z+...+w_N z^N+z^{N+1}$ are contained in the unit disk $\D$.
As in the proof of Lemma \ref{lem:J(w,p)-decomposition}, we then see that
for $z\in\mathbb{S}$,
\begin{equation}\label{eq:f/w^2}
\frac{f(z)}{w(z)^2}=\frac{f(z)}{z^{2N+2}}\,\prod_{j=1}^N\Big(\frac{1}{1-z_j/z}\Big)^2
=\frac{f(z)}{z^{2N+2}}\,g(z),
\end{equation}
where $g\in \HH^{1,-}$ and
\[
g(z)=1+2\,\big(\sum_{j=1}^N z_j\big)\,\frac{1}{z}+O\Big(\frac{1}{z^2}\Big),\quad z\to\infty.
\]
In view of \eqref{eq:w-polynomial} we have that $\sum_{j=1}^N z_j=-w_N$.

\begin{Definition}\label{def:deg}
For $f\in\HH^1$ and an integer $d\in\mathbb{Z}$ we write that $\deg f= d$ if 
$f=\sum_{k\le d} f_k e_k$ and $f_d\ne 0$. We set $\deg f=\infty$ if 
for any integer $\ell\ge 1$ there exists an index $n_\ell\ge\ell$ such that
$f_{n_\ell}\ne 0$.
\end{Definition}

We have the following Lemma. 

\begin{Lemma}\label{lem:Q_on_F_N}
Take $f\in\HH^1$ and assume that $w\in\mathcal{F}_N$ so that the zeros of 
the polynomial $w(z)=w_1 z+...+w_N z^N+z^{N+1}$ are contained in the unit disk $\D$.
Then we have:
\begin{itemize}
\item [(i)] If $\deg f\le 2N+1$ then $Q_f(w)=0$. More generally, 
$Q_f(w)$ does not depend on $f_k$, $k\le 2N+1$.
\item[(ii)] If $\deg f=2N+3$ then
\[
Q_f(w)=(f_{2N+2}-2f_{2N+3}w_N)+f_{2N+3}\,z\,.
\]
More generally, for $2N+2\le\deg f\not=\infty$ then $Q_f(w)$ is a polynomial of degree
$\deg f-2N-2$ with coefficients depending on $w_1,...,w_N$ and $f_k$, $2N+2\le k\le\deg f$.
\end{itemize}
\end{Lemma}

\noindent The proof of this Lemma follows directly from \eqref{eq:f/w^2} and thus will be omitted.
By combining Lemma \ref{lem:Q_on_F_N} with Proposition \ref{prop:finite_dynamics} we obtain 
the following Corollary.

\begin{Corollary}\label{coro:finite_image}
Take $f\in\HH^1$, $N\ge 2$, and let $(w,p)\in C^\ell\big([0,T),\OO\times\HH^{1,+}_0\big)$, 
$\ell\ge 2$, $T>0$, be the solution \eqref{eq:(w,p)-solution} of the Hamiltonian equation with 
Hamiltonian $H_f$ and initial data $(w_0,p_0)\in\mathcal{HF}_N$. Then, if $2N+2\le\deg f\not=\infty$ we 
have that $q_f(t)=Q_f(w(t))$ is a polynomial of degree $\le \deg f-2N-2$ for any $t\in[0,T)$. 
The same holds with $H_f$ replaced by $J_k^f$, $1\le k\le N$.
\end{Corollary}

\medskip

\noindent{\em Relation to geodesic equivalence.}
Here we will restrict ourselves to the case when $(w_1,...,w_N)$ and
$(p_1,...,p_N)$ are {\em real} and $f=0$.
In view of \eqref{eq:canonical_poisson_N} we can then identify the Poisson manifold 
$(\mathcal{PF}_N,\{\cdot,\cdot\}_N)$ with
the co-tangent bundle $T^*\R^N$ where $\{(w_1,...,w_N)\}$ are the coordinates on $\R^N$,
$\{(p_1,...,p_N)\}$ are the corresponding impulses, and $\omega_N:=\sum_{1\le j\le N}dp_j\wedge dw_j$ is
the canonical symplectic structure on $T^*\R^N$. The restriction of 
the Hamiltonian \eqref{eq:Hamiltonian} to $\mathcal{PF}_N$ takes the form
\begin{equation}\label{eq:H_N(w,p)}
H_N(w,p)=\sum_{1\le j,k\le N} h^{jk}(w)p_jp_k
\end{equation}
where the $N\times N$-matrix $\Gamma_N(w)=(h^{jk}(w))_{1\le j,k\le N}$ equals the Hankel matrix
\begin{equation}\label{eq:H_N(w)}
\Gamma_N(w)= \left(\begin{array}{ccccc}
     w_2 & w_3 & \dots & w_N & 1  \\
     w_3 &  & \iddots & 1 & 0  \\
     \vdots & \iddots & \iddots & \iddots &  \vdots\\
     w_N & 1 & \iddots &  & 0 \\ 
    1 & 0 & \dots & 0 & 0 \\ 
\end{array}\right).
\end{equation}
Note that $\Gamma_N(w)$ is a non-degenerate matrix with determinant equal to plus or minus one.
By applying the Legendre transform corresponding to the Hamiltonian $\frac{1}{2}\,H_N(w,p)$,
\begin{equation}\label{eq:LF_H_N}
FL_{H_N} : T^*\R^N\to T\R^N,\quad 
(w,p)\mapsto(w,v),\quad v=\frac{1}{2}\frac{\partial H_N}{\partial p}(w,p)=
\Gamma_N(w)\,p,
\end{equation}
we obtain from \eqref{eq:H_N(w,p)} a (pseudo) Riemannian metric on $\R^N$,
\begin{equation}\label{eq:g}
dg_N^2:=\sum_{1\le j,k\le N}g_{ik}(w)\,dw_j dw_k,
\end{equation}
whose metric tensor is given by the inverse of \eqref{eq:H_N(w)}.
Consider in addition the symmetric with respect to \eqref{eq:g} tensor field $M_N(w)$ of type $(1,1)$ 
on $\R^N$ that, in coordinates, is given by the matrix
\begin{equation}\label{eq:M-matrix}
M_N(w):= \left(\begin{array}{ccccc}
     0 & 0 & \dots & 0& w_1  \\
     -1 & 0 & \dots & 0 & w_2 \\
     0 & -1 & \dots & 0 & w_3 \\
     \vdots & \vdots & \ddots & \vdots &  \vdots\\
     0 & 0 & \dots & -1 & w_N \\ 
\end{array}\right)\,.
\end{equation}
One easily sees that the determinant of \eqref{eq:M-matrix} is equal to $w_1$.
Let us now assume that $w_1\ne 0$ and define a second metric on $\R^N$ by the formula
\begin{equation}\label{eq:g-bar}
\bar{g}_N(v,v):=\frac{1}{\det M_N(w)}\,g_N\big(M_N(w)^{-1}v,v\big),\quad v\in T\R^N\,.
\end{equation}
Recall that two Riemannian (or pseudo Riemannian) metrics $g$ and $\bar{g}$ on 
a smooth manifold $X$ of dimension $n\ge 2$ are called {\em geodesically equivalent} (or, equivalently, 
{\em projectively equivalent}) if they have the same unparametrized geodesics, i.e.,
if $\gamma\in C^\infty\big((a,b),X\big)$ is a geodesic of $g$ then 
$\gamma\circ\phi\in C^\infty\big((a_1,b_1),X\big)$, where $\phi : (a_1,b_1)\to(a,b)$ is 
a $C^\infty$-diffeomorphism, is a geodesic of $\bar{g}$, and vise versa. 
The following Lemma is well known and can be proved by a direct computation.

\begin{Lemma}
The metric \eqref{eq:g} and \eqref{eq:g-bar} are geodesically equavalent.
\end{Lemma}

\noindent The geodesic flows of geodesically equivalent metrics have integrals in involution.
More specifically, we have the following Theorem.
For a given metric $g$ on $X$ consider the Legendre transform $FL_g : TX\to T^*X$, 
$v\mapsto g(v,\cdot)$, and denote by $\omega$ the canonical symplectic form on $T^*X$.

\begin{Theorem}[\cite{MT1},\cite{T1}]
Assume that the Riemannian (or pseudo Riemannian) metrics $g$ and $\bar{g}$ are geodesically 
equivalent on a smooth manifold $X$ of dimension $n\ge 2$ and consider the endomorphism
$M$ of the tangent bundle $TX$ that, in coordinates, is given by\footnote{For repeated indexes we 
follow the Einstein summation convention.}
\begin{equation}\label{eq:M-operator_finite}
M^j_k:=\left|\frac{\det\bar{g}}{\det g}\right|^{\frac{1}{n+1}}\bar{g}^{jl}g_{lk},
\end{equation}
and let $\text{\rm Id}$ be the identity endomorphism of $TX$.
Then, for any given value of the parameter $\mu\in\R$ the function on the tangent bundle 
\begin{equation}\label{eq:I(mu)-finite}
I_g(\mu)(v):=\det(\mu\,\text{\rm Id}+M)\,g\big(\big(\mu\,\text{\rm Id}+M\big)^{-1}v,v\big),\quad v\in TX,
\end{equation}
is an integral of the geodesic flow of 
the metric $g$. Moreover, for any $v\in TX$ we have the expansion in the parameter $\mu\in\R$,
\begin{equation}\label{eq:I_k-finite}
I_g(\mu)(v)=I_{g,1}(v)+I_{g,2}(v)\,\mu+...+I_{g,n}(v)\,\mu^{n-1},
\end{equation}
where the functions $I_{g,k} : TX\to\R$, $1\le k\le n$, are in involution with respect to
to the symplectic structure $\omega_g:=FL_g^*\omega$ on $TX$.
\end{Theorem}

\noindent The integral $I_{g,n}(v)$ in \eqref{eq:I_k-finite} equals $g(v,v)$.
When written on the co-tangent bundle, the one-parameter family of integrals 
\eqref{eq:I(mu)-finite} takes the form
\begin{equation}\label{eq:I(mu)-finite'}
I_g(\mu)(p)=\det(\mu\,\text{\rm Id}+M)\,\big\1\big(\mu\text{\rm Id}+M\big)^{-1} g^{-1} p,p\big\2_X,
\quad\mu\in\R,\quad p\in T^*X,
\end{equation}
where $\1\cdot,\cdot\2_X$ is the standard pairing between vectors and co-vectors and
$g^{-1} : T^*X\to TX$ denotes the inverse of the Legendre transform $FL_g : TX\to T^*X$.

Let us now return to the particular case of the metrics \eqref{eq:g} and \eqref{eq:g-bar}.
A direct computation shows that the tensor field \eqref{eq:M-operator_finite} is then given by
the matrix \eqref{eq:M-matrix}. This allows us to compute the integrals \eqref{eq:I_k-finite}.
To this end, consider the adjoint matrix of $\mu\,\text{\rm Id}+M_N(w)$,
\begin{equation}\label{eq:A_N(mu)}
A_N(\mu,w):=\det(\mu\,\text{\rm Id}+M_N(w))\,
\big(\mu\,\text{\rm Id}+M_N(w)\big)^{-1},\quad\mu\in\R,\quad w\in\R^N.
\end{equation}
We first note that
\[
\det(\mu\,\text{\rm Id}+M_N(w))=w_1+w_2\,\mu+...+w_N\,\mu^{N-1}+\mu^N\,.
\]
A direct computations shows that
\begin{equation}\label{eq:A_{N,k}}
A_N(\mu,w)=A_{N,1}(w)+A_{N,2}(w)\,\mu+...+A_{N,N}(w)\,\mu^{N-1}
\end{equation}
where
\begin{equation*}\label{eq:A_{N,1}-matrix}
A_{N,1}(w)= \left(\begin{array}{cccccc}
     w_2 & -w_1 & 0 & \dots & 0& 0  \\
     w_3 & 0 & -w_1 & \dots & 0 & 0 \\
     w_4 & 0 & 0 &\dots & 0 & 0 \\
     \vdots & \vdots & \ddots & \vdots &  \vdots\\
     w_N & 0 & 0 & \dots & 0 & -w_1 \\
     1 & 0 & 0 &\dots & 0 & 0 \\ 
\end{array}\right),
\end{equation*}
\begin{equation*}\label{eq:A_{N,2}-matrix}
A_{N,2}(w)= \left(\begin{array}{ccccccc}
     w_3 & 0 & -w_1 & 0 & \dots & 0& 0  \\
     w_4 & w_3 & -w_2 & -w_1& \dots & 0 & 0 \\
     w_5 & w_4 & 0 &-w_2 & \dots & 0 & 0 \\
     \vdots & \vdots & \vdots & \vdots &  \ddots & \ddots & \vdots\\
     w_N & w_{N-1} & 0 & 0 &\dots & -w_2 & -w_1 \\
     1 & w_N & 0 & 0 &\dots & 0 & -w_2 \\
     0 & 1 & 0 & 0 &\dots & 0 & 0 \\ 
\end{array}\right),
\end{equation*}
and $A_{N,N}(w)$ is the identity $N\times N$-matrix. 
In particular, we obtain from \eqref{eq:I(mu)-finite'}, \eqref{eq:A_N(mu)},
and \eqref{eq:A_{N,k}}, the following formula for the integrals of \eqref{eq:g},
\begin{equation}\label{eq:I_{N,k}}
I_{N,k}(w,p)=\big\1 A_{N,k}(p)\Gamma_N(w) p,p\big\2_{\R^N},\quad 1\le k\le N,
\end{equation}
where $w,p\in \R^N$ and $\1 p,v\2_{\R^N}:=\sum_{1\le j\le N} p_j v_j$ for $v,p\in\R^N$.
By comparing the formulas for $A_{N,k}(w)$, $1\le k\le N$, with \eqref{eq:A_k-shifts} we 
conclude that
\begin{equation}\label{eq:adjoint_operators_equality_k}
A_k(w)v=A_{N,k}(w)v,\quad 1\le k\le N,
\end{equation}
for any $w\in\mathcal{F}_N$ and $v=v_1 e_1+...+v_N e_N\in\HH^{1,+}_0$ where we
identify $\mathcal{F}_N$ with $\R^N=\{(w_1,...,w_N)\}$.
Equivalently, \eqref{eq:adjoint_operators_equality_k} can be written as
\begin{equation}\label{eq:adjoint_operators_equality}
A(\mu,w)v=\mu\,A_N(\mu,w)v
\end{equation}
for any $w\in\mathcal{F}_N$, $v=v_1 e_1+...+v_N e_N\in\HH^{1,+}_0$, and $\mu\in\R$.

\begin{Remark}
The equality \eqref{eq:adjoint_operators_equality} implies that the geodesic operator
\eqref{eq:A-expansion(H)} is an infinite dimensional analog of the adjoint matrix 
\eqref{eq:A_N(mu)}. Note that we do {\em not} have an analog of the metric \eqref{eq:g}
and the linear transformation \eqref{eq:M-matrix} on $\HH^{1,+}_0$. 
The reason is that in contrast to the Legendre transform \eqref{eq:LF_H_N} on $T^*\R^N$ 
the Legendre transform \eqref{eq:Legendre_transform} on $\HH^{1,+}_0\times\HH^{1,+}_0$ is 
not invertible (see Remark \ref{rem:Legendre_transform}).
\end{Remark}

\noindent In particular, we obtain the following geometric interpretation of 
the Hamiltonian flow with Hamiltonian $\widetilde{H}_f$ on $\mathcal{PF}_N$ 
(see Proposition \ref{prop:finite_dynamics}).

\begin{Corollary}\label{coro:geodesic_equivalence}
Assume that $f=0$.
The Hamiltonian flow of $\widetilde{H}_f$ on $\mathcal{PF}_N$ coincides with
the geodesic flow of the metric \eqref{eq:g} when written on the co-tangent bundle
$T^*\R^N$. Moreover, the integrals \eqref{eq:J_f-tilde} coincide with the
integrals \eqref{eq:I(mu)-finite'}.
\end{Corollary}

\begin{proof}[Proof of Corollary \ref{coro:geodesic_equivalence}]
The first statement of the Corollary follows from the discussion above
and the fact that $\widetilde{H}_f(w,p)=H_N(w,p)$ for 
$(w,p)\in\mathcal{PF}_N$. Let us prove the second statement.
It follows from \eqref{eq:I_k} and \eqref{eq:adjoint_operators_equality_k} that 
for any $(w,p)\in\mathcal{PF}_N$ we have that
\[
I_k(w,p)=\big\1 A_k(w)\Gamma(w) p,p\big\2=\big\1 A_{N,k}(w)\Gamma_N(w) p,p\big\2_{\R^N}
=I_{N,k}(w,p),\quad 1\le k\le N,
\]
where we also used \eqref{eq:I_{N,k}} and the fact that $\Gamma(w)p=\Gamma_N(w)p$
for $(w,p)\in\mathcal{PF}_N$. This completes the proof of the Corollary.
\end{proof}

\noindent For the case when $f\ne 0$ we refer, e.g., to \cite{T2}.
We will not include the details here.

\begin{Remark}
Note that the definitions and the statements proved in this Section continue 
to hold with $\HH^{1,+}_0$ replaced by $\W^{1,+}_{r,0}$.
\end{Remark}

\section{$Q$-transforms and a class of PDEs}\label{sec:PDEs}
In this Section we construct two nonlinear transforms (the $Q$-transforms) that map the flows 
generated by the Hamiltonian $H_f$ and $J^f(\mu)$ for given $f\in\HH^1$ and $\mu\in\D$
onto solutions of a class of evolution PDE's.
To this end, we will first derive an equation on the evolution of the $w$-component
of the Hamiltonian flow of $H_f$ for a fixed value of the momentum map \eqref{eq:J-momentum_map(H)}. 
Take $f\in\HH^1$ and recall from Theorem \ref{th:existence} that
for any given $(w_\bullet,p_\bullet)\in\OO\times\HH^{1,+}_0$ there exists 
an open neighborhood $\mathcal{V}$ of $(w_\bullet,p_\bullet)$ in $\OO\times\HH^{1,+}_0$ and $T>0$ 
such that for any initial data $(w_0,p_0)\in\mathcal{V}$ there exists a unique solution 
\begin{equation}\label{eq:(w,p)-solution_bis}
(w,p)\in C^\ell\big([0,T),\OO\times\HH^{1,+}_0\big),\quad\forall\ell\ge 2,
\end{equation}
of the Hamiltonian equation with Hamiltonian $H_f$ such that $(w,p)|_{t=0}=(w_0,p_0)$.
Recall from \eqref{eq:Q_f} that for $w\in\OO$ we set $Q_f(w)=\Pi^+(f/w^2)$.
We have the following Theorem.

\begin{Theorem}\label{th:w-evolution}
Take $f\in\HH^1$ and let $(w,p)\in C^\ell\big([0,T),\OO\times\HH^{1,+}_0\big)$, $T>0$, $\ell\ge 2$,
be the solution of the Hamiltonian equation with Hamiltonian $H_f$ with 
initial data $(w_0,p_0)\in\OO\times\HH^{1,+}_0$. Then we have that
\begin{equation}\label{eq:w-evolution}
w\ddot{w}-\frac{1}{2}\dot{w}^2=
-2 w^2\,Q_f(w)-2 S_+J^f(w_0,p_0),\quad Q_f(w)=\Pi^+(f/w^2),
\end{equation}
where $J^f$ is the momentum map \eqref{eq:J-momentum_map(H)}, 
$S_+ : \HH^{1,+}_0\to\HH^{1,+}_0$ is the right shift \eqref{eq:S_pm}, and 
$\Pi^+ : \HH^1\to\HH^{1,+}$ is the projection \eqref{eq:Pi(H)}.
\end{Theorem}

\begin{proof}[Proof of Theorem \ref{th:w-evolution}]
It follows from \eqref{eq:U^f-gradient} and \eqref{eq:X_H_f} that 
the Hamiltonian equation of $H_f$ is
\begin{equation}\label{eq:X_H-tilde}
\left\{
\begin{array}{l}
\dot{w}=2\Pi^+_0(w\check{p}),\\
\dot{p}=-p^2+\Pi^+_0(\check{f}/\check{w}^2).
\end{array}
\right.
\end{equation}
We then obtain from \eqref{eq:X_H-tilde} that
\begin{align*}
\ddot{w}&=2\Pi^+_0(\dot{w}\check{p})+2\Pi^+_0\big(w\,(\dot{p})^\vee\big)=
4\Pi^+_0\big(\check{p}\Pi^+_0(w\check{p})\big)+
2\Pi^+_0\Big(-w\check{p}^2+w\big(\Pi^+_0(\check{f}/\check{w}^2)\big)^\vee\Big)\\
&=4\Pi^+_0\big(\check{p}\big(w\check{p}-\Pi^-(w\check{p})\big)\big)+
2\Pi^+_0\Big(-w\check{p}^2+w\big(\Pi^-(f/w^2)-\1 f/w^2,1\2\big)\Big)\\
&=2\Pi^+_0(w\check{p}^2)+2\Pi^+_0\big(w\Pi^-(f/w^2)\big)-2\1 f/w^2,1\2\,w\\
&=2\Pi^+_0(w\check{p}^2)+2\Pi^+_0\big(w\big(f/w^2-\Pi^+_0(f/w^2)\big)\big)-2\1 f/w^2,1\2\,w\\
&=\Big(2\Pi^+_0(w\check{p}^2)+2\Pi^+_0(f/w)\Big)-2w\Pi^+(f/w^2)
\end{align*}
where we used that 
$\big(\Pi^+_0a\big)^\vee=\Pi^-\check{a}-\1 a,1\2$ for any $a\in L^2(\T)$.
This together with \eqref{eq:J(w,p)} implies that
\begin{align}\label{eq:w''}
w\ddot{w}&=2\Big(w\Pi^+_0(w\check{p}^2)+w\Pi^+_0(f/w)\Big)-2w^2\Pi^+(f/w^2)\nonumber\\
&=2 \big(\Pi^+_0(w\check{p})\big)^2-2 S_+J^f(w,p)-2w^2\Pi^+(f/w^2)\,.
\end{align}
Theorem \ref{th:w-evolution} then follows from the first equation in \eqref{eq:X_H-tilde}, 
\eqref{eq:w''}, and the fact that $J^f(w,p)$ is independent on $t\in[0,T)$.
\end{proof}

\medskip

For a given $f\in\HH^1$ consider the maps
\begin{equation}\label{eq:Q+}
\mathcal{Q}_f^+ : \OO\times\HH^{1,+}_0\to\HH^{1,+},\quad
(w,p)\mapsto Q_f^+(w):=Q_f(w)=\Pi^+(f/w^2),
\end{equation}
and 
\begin{equation}\label{eq:Q-}
\mathcal{Q}_f^- : \OO\times\HH^{1,+}_0\to\HH^{1,-}_0,\quad
(w,p)\mapsto Q_f^-(w):=-\Pi^-_0(f/w^2)\,.
\end{equation}
It follows from \eqref{eq:Q+} and \eqref{eq:Q-} that
\begin{equation}\label{eq:Q+<->Q-}
Q_f^+(w)=Q_f^-(w)+f/w^2,\quad w\in\OO\,.
\end{equation}
Each of the quantities $\mathcal{Q}_f^+(w,p)$ and $\mathcal{Q}_f^-(w,p)$ 
will be called a {\em $Q$-transform} of $(w,p)\in\OO\times\HH^{1,+}_0$.
The Lemma below follows directly from Lemma \ref{lem:1/w-analytic} and 
the Banach algebra property of $\HH^1$.

\begin{Lemma}
The maps $\mathcal{Q}_f^\pm$ are analytic.
\end{Lemma}

\noindent It follows from \eqref{eq:Q+<->Q-} that the transformations 
$\mathcal{Q}_f^\pm : \OO\times\HH^{1,+}_0\to\HH^{1,\pm}$ 
map the solution \eqref{eq:(w,p)-solution_bis} of the Hamiltonian equation with Hamiltonian $H_f$ 
onto two curves
\begin{equation}\label{eq:q^pm(t)}
q^\pm(t):=\mathcal{Q}_f^\pm\big(w(t),p(t)\big)=Q_f^\pm\big(w(t)\big),\quad
q^\pm\in C^\ell\big([0,T],\HH^{1,\pm}\big),
\end{equation}
that satisfy the relation $q^+(t)=q^-(t)+f/w(t)^2$.
This and Theorem \ref{th:w-evolution} imply that
the $w$-component of the solution \eqref{eq:(w,p)-solution_bis} 
satisfies the equation
\begin{equation}\label{eq:w-evolution_bis}
w(t)\,\ddot{w}(t)-\frac{1}{2}\dot{w}(t)^2=
-2 w(t)^2\,q^\pm(t)-2C_f^\pm(w_0,p_0),\quad t\in[0,T),
\end{equation}
where $C_f^+(w_0,p_0)=S_+J^f(w_0,p_0)$ and $C_f^-(w_0,p_0)=S_+J^f(w_0,p_0)+f$ are 
independent of $t$. 

It follows from Theorem \ref{th:J(mu)-involutivity} and Theorem \ref{th:existence} that 
for any given $f\in\HH^1$ and $\mu\in\D$ the Hamiltonian flows of $H_f$ and $J^f(\mu)$ are 
locally uniquely defined on $\OO\times\HH^{1,+}_0$ and commute.
More specifically, we have that for any given $(w_\bullet,p_\bullet)\in\OO\times\HH^{1,+}_0$ 
there exists an open neighborhood $\mathcal{V}$ of 
$(w_\bullet,p_\bullet)$ in $\OO\times\HH^{1,+}_0$ and $T>0$ 
such that for any initial data $(w_0,p_0)\in\mathcal{V}$ there exists a $C^\ell$-map,
$\forall\ell\ge 2$,
\begin{equation}\label{eq:(w,p)-solution(x,t)}
(w,p) : [0,T)\times[0,T)\to\OO\times\HH^{1,+}_0,\quad(x,t)\mapsto\big(w(x,t),p(x,t)\big),
\end{equation}
such that $(w,p)|_{(x,t)=(0,0)}=(w_0,p_0)$ and for any given $t\in[0,T)$ the curve 
\[
[0,T)\to\OO\times\HH^{1,+}_0,\quad x\mapsto\big(w(x,t),p(x,t)\big),
\] 
is the unique solution of the Hamiltonian equation with Hamiltonian $H_f$ and 
initial data $\big(w(0,t),p(0,t)\big)$, and for any given $x\in[0,T)$ the curve 
\[
[0,T)\to\OO\times\HH^{1,+}_0,\quad t\mapsto\big(w(x,t),p(x,t)\big),
\]
is the unique solution of the Hamiltonian equation with Hamiltonian $J^f(\mu)$
and initial data $\big(w(x,0),p(x,0)\big)$. 
We will call the map \eqref{eq:(w,p)-solution(x,t)} the {\em joint flow} of $H_f$ and $J^f(\mu)$ 
with initial data $(w_0,p_0)$.
For $x,t\in[0,T)$ consider the $Q$-transforms of the joint flow \eqref{eq:(w,p)-solution(x,t)}
\begin{equation}\label{eq:q^pm(x,t)}
q^\pm(x,t):=\mathcal{Q}_f^\pm\big(w(x,t),p(x,t)\big).
\end{equation}
By construction, the maps 
\begin{equation}\label{eq:q^pm(x,t)-map}
q^+ : [0,T)\times[0,T)\to\HH^{1,+}\quad\text{\rm and}\quad
q^- : [0,T)\times[0,T)\to\HH^{1,-}_0
\end{equation}
belong to $C^\ell$ for any $\ell\ge 2$.
It follows from \eqref{eq:w-evolution_bis} that the $w$-component of the joint flow 
\eqref{eq:(w,p)-solution(x,t)} 
satisfies the equation\footnote{Note that the variable $t$ in \eqref{eq:w-evolution_bis} is now 
denoted by $x$.}
\begin{equation}\label{eq:w-evolution(x)}
w w_{xx}-\frac{1}{2}w_x^2=-2 w^2\,q^\pm-2C_f^\pm
\end{equation}
where $w_x$ denotes the partial derivative with respect to $x$, and $C_f^\pm$ are
independent of $x$.
By differentiating \eqref{eq:w-evolution(x)} with respect to $x$, we obtain that
\begin{equation}\label{eq:w-evolution(x)'}
w_{xxx}+4 w_x q^\pm+2w\,q^\pm_x=0\,.
\end{equation}
(It is no surprise that this equation appears in the theory of Hill's equation -- 
see \cite[\S\,3.4]{MW}.)
In what follows, we will derive a partial differential equation for the $t$-evolution of 
\eqref{eq:q^pm(x,t)}. Recall from Section \ref{sec:potentials} that for any given $\mu\in\D$ 
the function
\[
\widetilde{\delta}_\mu : \D\to\C,\quad \widetilde{\delta}_\mu(z)=\frac{\mu z}{1-\mu z},
\]
belongs to $\HH^{1,+}_0$ and satisfies \eqref{eq:tilde_delta3} where $\delta_\mu$ is 
the Dirac delta on $\D$. In addition to $\widetilde{\delta}_\mu$ with $\mu\in\D$ we will also need 
the function $\mathbb{1}_\mu\in\HH^{1,-}$ where
\begin{equation}\label{eq:1_mu}
\mathbb{1}_\mu(z):=\widecheck{\widetilde{\delta}}_\mu(z)/\mu=1/(z-\mu),\quad z\in\C\setminus\{\mu\}\,.
\end{equation}
We have the following alternative formula for the geodesic operator \eqref{eq:A-expansion(H)} in
$\HH^{1,+}_0$ and its adjoint with respect to the pairing \eqref{eq:L^2-metric(H0)} 
(cf. \eqref{eq:A_k-adjoint}).

\begin{Lemma}\label{lem:A-alternative_formula}
For any given $w,g\in\HH^{1,+}_0$ and $\mu\in\D$ we have that
\[
A(\mu,w)g=
g(\mu)\Gamma(w)(\widetilde{\delta}_\mu/\mu)-w(\mu)\Gamma(g)(\widetilde{\delta}_\mu/\mu)
\quad\text{\rm and}\quad
A(\mu,w)g=
\big(g(\mu)\,w-w(\mu)\,g\big)\,\mathbb{1}_\mu\,.
\]
\end{Lemma}

\begin{Remark}\label{rem:no_pole}
Let us briefly comment on the second formula in Lemma \ref{lem:A-alternative_formula}.
The function $\mathbb{1}_\mu(z)=1/(z-\mu)$ has a pole of order one at 
$\mu\in\D$ and belongs to $\HH^{1,-}$. However, the difference 
$\big(g(\mu)\,w-w(\mu)\,g\big)\,\mathbb{1}_\mu$ belongs to $\HH^{1,+}_0$ since
the analytic function $g(\mu)\,w-w(\mu)\,g$ has a zero at $\mu$.
\end{Remark}

\begin{proof}[Proof of Lemma \ref{lem:A-alternative_formula}]
For any $w,g\in\HH^{1,+}_0$ and $\mu,z\in\D$ we have that
\begin{align*}
&g(\mu)\,\big[\Gamma(w)(\widetilde{\delta}_\mu/\mu)\big](z)-
w(\mu)\,\big[\Gamma(g)(\widetilde{\delta}_\mu/\mu)\big](z)
=\Pi^+_0\Big(g(\mu) w(z) \widecheck{\widetilde{\delta}}_\mu(z)/\mu-
w(\mu) g(z) \widecheck{\widetilde{\delta}}_\mu(z)/\mu\Big)\\
&=\Pi^+_0\left(\frac{w(z) g(\mu)-w(\mu) g(z)}{z-\mu}\right)=
\frac{1}{z-\mu}\,\Big(w(z) g(\mu)-w(\mu) g(z)\Big)
\end{align*}
where we used that $\widecheck{\widetilde{\delta}}_\mu(z)/\mu=1/(z-\mu)$
and the fact that the expression on the right, when considered as a function of $z\in\D$, 
belongs to $\HH^{1,+}_0$ and coincides with $A(\mu,w)g$ 
(cf. Proposition \ref{prop:A-operator(H)} (iii), (iv)). 
This proves the first and the second formula in the Lemma. 
\end{proof}

\noindent By combining Lemma \ref{lem:A-alternative_formula} with \eqref{eq:nabla_pJ(mu)} we
obtain the following Lemma.

\begin{Lemma}\label{lem:w-evolution(x,t)}
For any given $f\in\HH^1$ and $\mu\in\D$ the $w$-component $w : [0,T)\times[0,T)\to\HH^{1,+}_0$ 
in \eqref{eq:(w,p)-solution(x,t)} satisfies the equation
\begin{equation*}
w_t=\big(w_x(\mu)\,w-w(\mu)\,w_x\big)\,\mathbb{1}_\mu
\end{equation*}
where $\mathbb{1}_\mu$ is given by \eqref{eq:1_mu}.
\end{Lemma}

\begin{proof}[Proof of Lemma \ref{lem:w-evolution(x,t)}]
Let $w : [0,T)\times[0,T)\to\HH^{1,+}_0$ be the $w$-component in 
\eqref{eq:(w,p)-solution(x,t)}. It then follows from \eqref{eq:nabla_pJ(mu)}, 
the first equation in \eqref{eq:X_H-tilde}, and Lemma \ref{lem:A-alternative_formula}, that
\begin{align*}
 w_t&=\nabla_pJ^f(\mu)=2A(\mu,w)\Gamma(w)p=A(\mu,w) w_x
 =\big(w_x(\mu)\,w-w(\mu)\,w_x\big)\,\mathbb{1}_\mu\,.
\end{align*}
This completes the proof of the Lemma.
\end{proof}

We are now ready to derive an evolution equation for the $Q$-transforms $q^\pm$ 
(cf. \eqref{eq:q^pm(x,t)-map}).
Consider the $C^\ell$-maps, $\ell\ge 3$,
\begin{equation}\label{eq:(w,q)-map}
w : [0,T)\times[0,T)\to\HH^{1,+}_0\quad\text{\rm and}\quad
q : [0,T)\times[0,T)\to\HH^{1,+}
\end{equation}
where the first map is the $w$-component of the joint flow \eqref{eq:(w,p)-solution(x,t)}
of $H_f$ and $J^f(\mu)$ and the second map is the $Q$-transform $q^+ : [0,T)\times[0,T)\to\HH^{1,+}$ 
of \eqref{eq:(w,p)-solution(x,t)} which, for simplicity, 
is denoted by $q$. We have the following Theorem.

\begin{Theorem}\label{th:(w,q)-evolution}
For any given $f\in\HH^1$ and $\mu\in\D$ the map 
$(w,q)\in C^\ell\big([0,T)\times[0,T),\HH^{1,+}_0\times\HH^{1,+}\big)$, $\ell\ge 3$,
given by \eqref{eq:(w,q)-map}, satisfies the system of equations 
\begin{equation}\label{eq:mu-equation}
\left\{
\begin{array}{l}
q_t+\big(w_{xxx}(\mu)+4w_x(\mu)\,q+2w(\mu)\,q_x\big)\,\mathbb{1}_\mu=0,\\
w_{xxx}+4 w_x q+2w\,q_x=0,
\end{array}
\right.
\end{equation}
where $\mathbb{1}_\mu$ is given by \eqref{eq:1_mu}.
The statement also holds with $q$ replaced by the map $q^- :[0,T)\times[0,T)\to\HH^{1,-}_0$
in \eqref{eq:q^pm(x,t)-map}.
\end{Theorem}

\noindent In this way, for any choice of $f\in\HH^{1}$ we obtain two types of 
solutions of \eqref{eq:mu-equation}: a ``plus'' solution,
\[
(w,q) : [0,T)\times[0,T)\to\HH^{1,+}_0\times\HH^{1,+},
\]
and a ``minus'' solution,
\[
(w,q^-) : [0,T)\times[0,T)\to\HH^{1,+}_0\times\HH^{1,-}_0\,.
\]
In view of \eqref{eq:Q+<->Q-}, these solutions are related to each other by 
the nonlinear transformation $q=q^-+f/w^2$\,.

\begin{Remark}
Since $w_{xxx}(\mu)+4w_x(\mu)\,q(\mu)+2w(\mu)\,q_x(\mu)=0$ the holomorphic in 
the variable $z\in\D$ function 
$z\mapsto w_{xxx}(\mu)+4w_x(\mu)\,q(z)+2w(\mu)\,q_x(z)$, $\D\to\C$,
has a zero at $z=\mu$ for any $(x,t)\in[0,T)\times[0,T)$.
In particular, the term
$\big(w_{xxx}(\mu)+4w_x(\mu)\,q+2w(\mu)\,q_x\big)\,\mathbb{1}_\mu$
in \eqref{eq:mu-equation} does not have a pole at $z=\mu$ and belongs to $\HH^{1,+}$.
\end{Remark}

\begin{proof}[Proof of Theorem \ref{th:(w,q)-evolution}]
Take $f\in\HH^1$ and $\mu\in\D$ and let $w : [0,T)\times[0,T)\to\HH^{1,+}_0$
and $q : [0,T)\times[0,T)\to\HH^{1,+}$ be as assumed in Theorem \ref{th:(w,q)-evolution}.
We then differentiate the equality \eqref{eq:w-evolution(x)} with respect to $t$ to
obtain that
\begin{equation}\label{eq:w-relation1}
w_t w_{xx}+w (w_t)_{xx}-w_x (w_t)_x=-4 w w_t q-2 w^2 q_t
\end{equation}
where we used the smoothness of the map \eqref{eq:(w,p)-solution(x,t)} to changed the order 
of differentiations. By Lemma \ref{lem:w-evolution(x,t)},
\begin{equation}\label{eq:w-relation2}
w_t=\big(w_x(\mu)\,w-w(\mu)\,w_x\big)\,\mathbb{1}_\mu\,.
\end{equation}
We then substitute \eqref{eq:w-relation2} into \eqref{eq:w-relation1} to obtain 
(after several cancellations) that
\begin{equation}\label{eq:wq-relation1}
-2 w\,q_t-4\big(w_x(\mu)\,w-w(\mu)\,w_x\big)\,q\,\mathbb{1}_\mu
=\big(w_{xxx}(\mu)\,w-w(\mu)\,w_{xxx}\big)\,\mathbb{1}_\mu\,.
\end{equation}
It now follows from \eqref{eq:wq-relation1} that
\begin{align}\label{eq:wq-relation2}
-2 w\,q_t&=\big(4 w_x(\mu) w q\,\mathbb{1}_\mu
-4 w(\mu) w_x q\,\mathbb{1}_\mu\big)
+\big(w_{xxx}(\mu) w\,\mathbb{1}_\mu-
w(\mu) w_{xxx}\,\mathbb{1}_\mu\big)\nonumber\\
&=\big(w_{xxx}(\mu)+4w_x(\mu)\,q+2 w(\mu)\,q_x\big)\,w\,\mathbb{1}_\mu
-\big( w_{xxx}+4 w_x\,q+2 w\,q_x\big)\,w(\mu)\,\mathbb{1}_\mu\nonumber\\
&=\big(w_{xxx}(\mu)+4w_x(\mu)\,q+2 w(\mu)\,q_x\big)\,
w\,\mathbb{1}_\mu\nonumber
\end{align}
where we used \eqref{eq:w-evolution(x)'} to conclude that the second term in brackets, that
appears in the second line, vanishes. This completes the proof of the Theorem in the case when
$q=q^+$. The case of $q^-$ is proved in the same way.
\end{proof}

Theorem \ref{th:(w,q)-evolution} allows us to obtain solutions of a number of
PDEs by choosing specific values of the parameter $f\in\HH^1$.
We illustrate this by considering the following Examples.

\noindent{\em Example (BKM systems, \cite{BKM1}).} 
These systems correspond to minus solutions of \eqref{eq:mu-equation} with a particular
choice of $f\in\HH^1$. More specifically, we take
\begin{equation}\label{eq:f_BKM}
f(z):=z^2/m(z)
\end{equation}
where $m(z)$ is a polynomial of degree $n\ge 1$ whose roots are contained in $\D$.
We will also assume that $w\in\OO_1$ where $\OO_1$ is the open set in $\HH^{1,+}_0$,
\[
\OO_1:=\big\{w\in\HH^{1,+}_0\,\big|\,
z=0\,\,\text{\em is the only zero of}\,\, w\,\, 
\text{\em in the closure of}\,\,\D\big\}\,.
\]
Assume that the polynomial $m(z)$ has zeros $z_1,...,z_s$, $z_j\ne z_k$ for $j\ne k$, 
with multiplicities $n_1,...,n_s$ respectively, $\sum_{1\le j\le s}n_j=n$.

We have the following Lemma.

\begin{Lemma}\label{lem:BKM}
Take a polynomial $m(z)$ of degree $n\ge 1$ as above and assume that $f\in\HH^1$ is 
of the form \eqref{eq:f_BKM}. Then we have:
\begin{itemize}
\item[(i)] For any choice of $w\in\OO_1$ the quantity
$\sigma=-m\,\Pi^-_0(f/w^2)$ is a polynomial of degree $\le n$ such that
\[
Q_f^-(w)\equiv-\Pi^-_0(f/w^2)=\sigma/m\in\HH^{1,-}_0\,.
\]
\item[(ii)] The map 
\[
\mathcal{S}_m : \HH^{1,+}_0\cap\OO_1\to\HH^{1,+},\quad w\mapsto\mathcal{S}_m(w):=-m\,\Pi^-_0(f/w^2),
\]
is analytic. In particular, the coefficients of the polynomial $\sigma$ in (i) depend analytically
on $w\in\HH^{1,+}_0\cap\OO_1.$
\end{itemize}
\end{Lemma}

\begin{proof}[Proof of Lemma \ref{lem:BKM}]
Take $f\in\HH^1$ of the form \eqref{eq:f_BKM} where the roots of the polynomial $m$, $\deg m=n$, 
are contained in $\D$. It then follows from Lemma \ref{lem:1/w} (iv) and the fact that
$w\in\OO_1$ that 
\begin{equation*}
\frac{f(z)}{w^2(z)}=\frac{g(z)}{\prod_{j=1}^s(z-z_j)^{n_j}}
\end{equation*}
for some $g\in\HH^{1,+}$, and hence, $f(z)/w(z)^2$ is a meromorphic function on $\D$ with poles at 
the zeros $z_1,...,z_s$ of the polynomial $m(z)$.
This implies that, up to a constant term,  $\Pi^-(f/w^2)$ is equal to the sum of 
the principal parts of $f(z)/w^2(z)$ at the poles $z_1,...,z_s$, and hence,
\begin{equation}\label{eq:principal_parts}
\big[\Pi^-_0(f/w^2)\big](z)=\sum_{k=1}^s\Big(\sum_{j=1}^{n_k}\frac{a_{kj}}{(z-z_k)^j}\Big)-C,
\end{equation}
where $a_{kj}$, $1\le k\le s$, $1\le j\le n_k$, are constants and $C$ is chosen so that
the right side of \eqref{eq:principal_parts} has vanishing $0$-th mode.
It then follows from \eqref{eq:principal_parts} that 
$\sigma(z)=-m(z)\,\big[\Pi^-_0(f/w^2)\big](z)$ is a polynomial of degree $\le n$. 
This proves item (i) of the Lemma.
In order to prove item (ii) we note that it follows from Lemma \ref{lem:1/w-analytic} and
the Banach algebra property of $\HH^1$ that the map
\[
\HH^{1,+}_0\cap\OO_1\to\HH^1,\quad w\mapsto -m\,\Pi^-_0(f/w^2),
\]
is analytic. In view of \eqref{eq:principal_parts}, the image of this map is a polynomial 
of degree $\le n$, and hence, it is contained in the closed subspace $\HH^{1,+}\subseteq\HH^1$.
This proves item (ii), and the Lemma.
\end{proof}

Let $(w,q_-)\in C^\ell\big([0,T)\times[0,T),\OO_1\times\HH^{1,-}_0\big)$, $\ell\ge 3$, be
a minus solution of \eqref{eq:mu-equation} given by Theorem \ref{th:existence} with initial data
$(w_0,p_0)\in\OO_1\times\HH^{1,-}_0$.
It then follows from Lemma \ref{lem:BKM} that $q_-=\sigma/m$ where 
$\sigma(x,t)=\mathcal{S}_m\big(w(x,t)\big)$, $x,t\in[0,T)$, is a polynomial of degree $\le n$ such that
\begin{equation*}
\sigma\in C^\ell\big([0,T)\times[0,T),\C^{n+1}\big)\,.
\end{equation*}
(Here, the space of polynomials of degree $\le n$ is identified with $\C^{n+1}$.)
The Corollary below then follows directly from Theorem \ref{th:existence}.

\begin{Corollary}\label{Coro:BKM}
For any choice of the polynomial $m(z)$ of degree $n\ge 1$ as above, $\mu\in\D$, $f\in\HH^1$ of 
the form \eqref{eq:f_BKM} the map 
$(w,\sigma)\in C^\ell\big([0,T)\times[0,T),\HH^{1,+}_0\times\C^{n+1}\big)$,
$\ell\ge 3$, is a solution of the BKM system
\begin{equation}\label{eq:BKM}
\left\{
\begin{array}{l}
\sigma_t+\big(m\,w_{xxx}(\mu)+4w_x(\mu)\,\sigma+2w(\mu)\,\sigma_x\big)\,\mathbb{1}_\mu=0,\\
m\,w_{xxx}+4 w_x\sigma+2w\,\sigma_x=0\,.
\end{array}
\right.
\end{equation}
\end{Corollary}

\noindent Theorem \ref{th:existence} allows us to construct solutions of \eqref{eq:BKM} where
$\sigma$ is not necessarily a polynomial but a generic element of $\HH^1$. 
In the Remark below we discuss this in more detail.

\begin{Remark}
For a given polynomial $m(z)$ of degree $n\ge 1$, Lemma \ref{lem:BKM} can easily be generalized to 
the case where $w\in\OO$ has $n_0\ge 1$ zeros inside the unit disk $\D$. In that case, $\sigma$ is
a fraction of two polynomials,
\begin{equation}\label{eq:P/Q}
\sigma(z)=\frac{P(z)}{Q(z)},\quad\deg P=2n_0+n,\quad\deg Q=2 n_0,
\end{equation}
whose coefficients depend analytically on $w\in\OO$.
This allows us to generalize Corollary \ref{Coro:BKM} to the case where
$\sigma\in C^\ell\big([0,T)\times[0,T),\HH^1\big)$ is of the form \eqref{eq:P/Q}
and satisfies \eqref{eq:BKM}. Note in addition, that the parameter function $f\in\HH^1$ 
can also be chosen to have an essential singularity in $\D$.
\end{Remark}

\noindent{\em Example (KdV-type systems).} 
These systems are related to the finite dimensional dynamics discussed in 
Section \ref{sec:finite_dynamics}.
As in \eqref{eq:(w,p)-solution(x,t)}, it follows from Theorem \ref{th:J_k-involutivity}
and Theorem \ref{th:existence} that for any given $f\in\HH^{1}$, $k\ge 1$, and
$(w_\bullet,p_\bullet)\in\OO\times\HH^{1,+}_0$,
there exists an open neighborhood $\mathcal{V}$ of $(w_\bullet,p_\bullet)$ in $\OO\times\HH^{1,+}_0$
and $T>0$ such that for any $(w_0,p_0)\in\mathcal{V}$ there exists a $C^\ell$-map, $\forall\ell\ge 2$,
\begin{equation}\label{eq:(w,p)-solution(x,t)_bis}
(w,p)\in C^\ell\big([0,T)\times[0,T),\OO\times\HH^{1,+}_0\big),
\end{equation}
such that $(w,p)|_{(x,t)=(0,0)}=(w_0,p_0)$  and for any given $t\in[0,T)$ the curve
\[
[0,T)\to\OO\times\HH^{1,+}_0,\quad x\mapsto\big(w(x,t),p(x,t)\big),
\]
is the unique solution of the Hamiltonian equation with Hamiltonian $H_f$ and initial data
$\big(w(0,t),p(0,t)\big)$, and for any given $x\in[0,T)$ the curve
\[
[0,T)\to\OO\times\HH^{1,+}_0,\quad t\mapsto\big(w(x,t),p(x,t)\big),
\]
is the unique solution of the Hamiltonian equation with Hamiltonian $J^f_k$ with
initial data $\big(w(x,0),p(x,0)\big)$. This is the joint flow of $H_f$ and $J^f_k$.

For a given $N\ge 2$ we choose $f\in\HH^1$, $\deg f=2N+3$ (see Definition \ref{def:deg}), such that
\begin{equation}\label{eq:f_KdV}
f_{2N+3}=1\quad\text{\em and}\quad f_{2N+2}=0\,.
\end{equation}
We then consider the map \eqref{eq:(w,p)-solution(x,t)_bis} that corresponds to the Hamiltonian flows
of $H_f$ and $J^f_k$ with $k=N-1$, and initial data $(w_0,p_0)\in\mathcal{HF}_N$ 
(cf. Section \ref{sec:finite_dynamics} for the notation). By Proposition \ref{prop:finite_dynamics},
$w(x,t)\in\mathcal{F}_N$ for any $x,t\in[0,T)$. We choose $w_0$ and $T>0$ so that the zeros of 
the polynomial $w(x,t)$ are contained in the unit disk $\D$. It then follows from
Lemma \ref{lem:Q_on_F_N} that the $Q$-transform $q(x,t):=Q_f\big(w(x,t)\big)$, $x,t\in[0,T)$,
of \eqref{eq:(w,p)-solution(x,t)_bis} is of the form
\[
q(x,t)=z-2 w_N(x,t),\quad x,t\in[0,T),
\]
where $z\in\D$ is the variable associated with the space $\HH^{1,+}_0$, considered as
a space of holomorphic functions on the unit disk $\D$, and $w_N(x,t)$ is the coefficient of 
$w(x,t)$ in front of $z^N$. A direct computation then shows that the $0$-th mode 
$q_0=-2 w_N$ of $q$ satisfies the KdV equation
\begin{equation}\label{eq:KdV}
2q_{0t}+q_{0xxx}+6q_0q_{0x}=0
\end{equation}
for any $x,t\in[0,T)$. 

\begin{Remark}
The equation \eqref{eq:KdV} can be easily obtained formally, by comparing the coefficients
in front of $\mu^N$ in equation \eqref{eq:mu-equation}, written as
\[
(z-\mu)\,q_t(z)+w_{xxx}(\mu)+4 w_x(\mu)\,q(z)+2 w(\mu)\,q_x(z)=0,
\]
where we set $z=0$ and $q_t=\sum_{j=1}^Nq_{t_j}\mu^j$ where $q_{t_j}$ the time derivative
of $q$ with respect to the flow of $J^f_j$.
\end{Remark}

\noindent In view of Lemma \ref{lem:Q_on_F_N}, by increasing the degree of $f\in\HH^1$, 
$\deg f>2N+3$, and no longer assuming \eqref{eq:f_KdV}, we obtain vector analogs of 
the KdV equation. We will not discuss this case in detail here.

\begin{Remark}
Note that the definitions and the statements proved in this Section continue 
to hold with $\HH^{1,+}_0$ replaced by $\W^{1,+}_{r,0}$.
\end{Remark}

\appendix
\section{Appendix}
In this Appendix, for the convenience of the reader, we provide the proof of 
Lemma \ref{lem:W-space}.

\begin{proof}[Proof of Lemma \ref{lem:W-space}]
The proof of the Lemma is straightforward.
Denote by $\ell^1_r$ the linear space of complex-valued sequences $w=(w_k)_{k\in\Z}$
equipped with the weighted norm
\[
\|w\|_{\ell^1_r}:=\Big(\sum_{k\in\Z}|w_k|^2r^{2|k|}\1 k\2^2\Big)^{1/2}
\]
and let $\ell^2$ be the Hilbert space of complex-valued square summable sequences 
equipped with the norm $\|w\|_{\ell^2}:=\sum_{k\in\Z}|w_k|^2$.
Then, the map $\ell^2\to\ell^1_r$, $w\mapsto\big(w_kr^{-|k|}\1 k\2^{-1}\big)_{k\in\Z}$,
is an isometry. This proves that $\ell^1_r$ is a Hilbert space 
when equipped with the scalar product 
$(v,w)_{\ell^1_r}:=\sum_{k\in\Z}v_k\overline{w}_kr^{2|k|}\1 k\2^2$.
On the other side, it is clear that the map
\begin{equation}\label{eq:W->l}
\W^1_r\to\ell^1_r,\quad f\mapsto(f_k)_{k\in\Z},
\end{equation}
where $(f_k)_{k\in\Z}$ are the Laurent coefficient of the holomorphic function $f$,
is an isomorphism of linear spaces. In fact, since the Laurent coefficients characterize a holomorphic
function in an annulus uniquely, the map \eqref{eq:W->l} is injective.
In order to see that the map is onto we take an arbitrary element $w\in\ell^1_r$ and consider 
the series
\begin{equation}\label{eq:f(z)_Laurent_expansion}
f(z):=\sum_{k\ge 0}\frac{w_{(-k)}}{z^k}+
\sum_{k\ge 1}w_kz^k\,.
\end{equation}
where $1/r\le|z|\le r$. It follows from \eqref{eq:f(z)_Laurent_expansion} and
Cauchy-Schwarz inequality that for any $1/r\le|z|\le r$,
\begin{equation}\label{eq:L^infty-inequality}
|f(z)|\le\sum_{k\in\Z}|w_k| r^{|k|}\le\Big(\sum_{k\in\Z}\frac{1}{\1 k\2^2}\Big)^{1/2}
\Big(\sum_{k\in\Z}|w_k|^2 r^{2|k|}\1 k\2^2\Big)^{1/2}\le C\|w\|_{\W^1_r}
\end{equation}
where $C:=\Big(\sum_{k\in\Z}\frac{1}{\1 k\2^2}\Big)^{1/2}=\sqrt{1+\pi^2/3}\le 3$.
This implies that the series \eqref{eq:f(z)_Laurent_expansion} converges absolutely 
and uniformly to a holomorphic function $f$ on $\A_r$ that extends to a continuous function 
on the closure $1/r\le|z|\le r$. In particular, $f\in\W^1_r$. Hence, the map \eqref{eq:W->l} is 
an isomorphism of linear spaces, as claimed. The map \eqref{eq:W->l} is an isometry by 
the definition of the norm \eqref{eq:W1-norm}. This proves (i). Item (iii) then follows from 
\eqref{eq:L^infty-inequality} and the discussion above. Let us now prove (ii).
Take $f,g\in\W^1_r$ and consider the function $fg$. It is clear that $fg$ is holomorphic on $\A_r$.
Let us now estimate the norm \eqref{eq:W1-norm} of $fg$. Denote by $\ell^1$ the Banach space 
of absolutely summable complex-valued sequences $(w)_{k\in\Z}$ equipped with 
the norm $\|w\|_{\ell^1}:=\sum_{k\in\Z}|w_k|$. We have
\begin{align}\label{eq:product_estimate1}
\|fg\|_{W^1_r}&=\big\|\big((fg)_kr^{|k|}\1 k\2\big)_{k\in\Z}\big\|_{\ell^2}\le
\big\|\big(\sum_{l\in\Z}|f_{k-l}| |g_l|\,(r^{|k-l|}r^{|l|})(\1 k-l\2+\1 
l\2)\big)_{k\in\Z}\big\|_{\ell^2}\nonumber\\
&\le\big\|\big(\sum_{l\in\Z}\big(|f_{k-l}|r^{|k-l|}\1 k-l\2\big)
\big(|g_l|r^{|l|}\big)\big)_{k\in\Z}\big\|_{\ell^2}
+\big\|\big(\sum_{l\in\Z}\big(|f_{k-l}|r^{|k-l|}\big)\big(|g_l|r^{|l|}\1 
l\2\big)\big)_{k\in\Z}\big\|_{\ell^2}\nonumber\\
&\le\|f\|_{\W^1_r}\big\|\big(|g_l| r^{|l|}\big)_{l\in\Z}\big\|_{\ell^1}+
\|g\|_{\W^1_r}\big\|\big(|f_l| r^{|l|}\big)_{l\in\Z}\big\|_{\ell^1}
\end{align}
where we used that $|k|\le|k-l|+|l|$ and $\1 k\2\le\1 k-l\2+\1 l\2$ for any $k,l\in\Z$
as well as the Young's inequality $\|a\star b\|_{\ell^2}\le\|a\|_{\ell^2}\|b\|_{\ell^1}$
on the convolution of two sequences $a\in\ell^2$ and $b\in\ell^1$. 
The estimates in \eqref{eq:L^infty-inequality} imply that
\[
\big\|\big(|w_l| r^{|l|}\big)_{l\in\Z}\big\|_{\ell^1}\le C\|w\|_{\W^1_r},\quad \forall w\in\W^1_r,
\]
with the same constant $C>0$. Hence, in view of \eqref{eq:product_estimate1},
\[
\|fg\|_{W^1_r}\le 2C\|f\|_{\W^1_r}\|g\|_{\W^1_r},\quad \forall f,g\in\W^1_r.
\]
Hence, $W^1_r$ is a Banach algebra with respect to the pointwise multiplication of functions.
This completes the proof of the Lemma.
\end{proof}


\end{document}